\documentclass[10pt]{article}

\usepackage{amscd,amsfonts,amsmath,amssymb,amsthm}
\usepackage[section]{algorithm}
\usepackage[noend]{algpseudocode}
\usepackage[title]{appendix}
\usepackage[affil-it]{authblk}
\usepackage{boldtensors}
\usepackage{booktabs}
\usepackage{fancyhdr}
\usepackage{float}
\usepackage{graphicx}
\usepackage{ifpdf}
\usepackage{mathtools}
\usepackage[frozencache,cachedir=.]{minted}
\usepackage{soul}
\usepackage{subcaption}
\usepackage{url}
\usepackage{xcolor}
\usepackage{hyperref}
\usepackage[nameinlink]{cleveref}
\usepackage{siunitx}
\usepackage{enumitem}
\usepackage{subcaption}
\usepackage{tikz}

\ifpdf
\hypersetup{
    pdftitle={Inverse problems for history-enriched linear model reduction},
    pdfauthor={Arjun Vijaywargiya, George Biros}
}
\fi

\hypersetup{
    colorlinks=true,
    linkcolor=black,
    citecolor=black,
    urlcolor=black,
}

\crefformat{equation}{#2(#1)#3}
\numberwithin{equation}{section}

\theoremstyle{plain}
\newtheorem{theorem}{Theorem}

\newtheorem{lemma}{Lemma}

\theoremstyle{remark}
\newtheorem{remark}{Remark}

\theoremstyle{definition}
\newtheorem{defn}{Definition}

\numberwithin{theorem}{section}
\numberwithin{proposition}{section}
\numberwithin{lemma}{section}
\numberwithin{corollary}{section}
\numberwithin{remark}{section}
\numberwithin{defn}{section}

\crefname{theorem}{Theorem}{Theorems}
\Crefname{theorem}{Theorem}{Theorems}
\crefname{proposition}{Proposition}{Propositions}
\Crefname{proposition}{Proposition}{Propositions}
\crefname{lemma}{Lemma}{Lemmata}
\Crefname{lemma}{Lemma}{Lemmata}
\crefname{corollary}{Corollary}{Corollaries}
\Crefname{corollary}{Corollary}{Corollaries}
\crefname{algorithm}{Algorithm}{Algorithms}
\Crefname{algorithm}{Algorithm}{Algorithms}
\crefname{appendix}{Appendix}{Appendices}
\Crefname{appendix}{Appendix}{Appendices}

\setminted{
    style=pastie,
    fontsize=\normalsize,
    linenos,
    frame=none
}

\begin{document} 

\pdfinfo{
   /Author (Arjun Vijaywargiya, George Biros)
   /Title (Inverse problems for history-enriched linear model reduction)
   /Keywords (Mori-Zwanzig formalism, closure modeling, memory closure, model reduction, data driven model, scientific machine learning)
}

\title{\textbf{Inverse problems for history-enriched linear model reduction}}

\author[1]{Arjun~Vijaywargiya\thanks{Corresponding author. E-mail: \href{mailto:arjunveejay@utexas.edu}{arjunveejay@utexas.edu}.}}
\author[1]{George~Biros}

\affil[1]{\normalsize Oden Institute for Computational Engineering and Sciences, The University of Texas at Austin}

\date{}

\maketitle

\vspace{-.25in}

\begin{abstract}
\noindent
Standard projection-based model reduction for dynamical systems incurs closure error because it only accounts for instantaneous dependence on the resolved state. From the Mori-Zwanzig (MZ) perspective, projecting the full dynamics onto a low-dimensional resolved subspace induces additional noise and memory terms arising from the dynamics of the unresolved component in the orthogonal complement. The memory term makes the resolved dynamics explicitly history dependent. In this work, based on the MZ identity, we derive exact, history-enriched models for the resolved dynamics of linear driven dynamical systems and formulate inverse problems to learn model operators from discrete snapshot data via least-squares regression. We propose a greedy time-marching scheme to solve the inverse problems efficiently and analyze operator identifiability under full and partial observation data availability. For full observation data, we show that, under mild assumptions, the operators are identifiable even when the full-state dynamics are governed by a general time-varying linear operator, whereas with partial observation data the inverse problem has a unique solution only when the full-state operator is time-invariant. To address the resulting non-uniqueness in the time-varying case, we introduce a time-smoothing Tikhonov regularization. Numerical results demonstrate that the operators can be faithfully reconstructed from both full and partial observation data and that the learned history-enriched MZ models yield accurate trajectories of the resolved state.
\end{abstract}

\noindent
\textbf{Keywords}: Mori-Zwanzig formalism, closure modeling, memory closure, model reduction, data driven model, scientific machine learning

\section{Introduction} \label{s:intro}
In many areas of science and engineering, one encounters driven linear dynamical systems of the form
\begin{align}\label{eq:cont_dyn}
\dot{~u}(t) = ~A ~u(t) + \bar{~g}(t), \qquad ~u(0) = ~u_0, \qquad t \in [0,T],
\end{align}
where $~u(t) \in \mathbb{R}^N$ is the (typically high dimensional) state vector, $~A$ is a linear operator, $\bar{~g}(t)$ is a prescribed forcing vector, and $T>0$ is the time horizon. A classical problem in machine learning is the recovery of an effective operator $\hat{~A}$ from discrete snapshot data that best explains the observed dynamics. Widely used methods like Dynamic Mode Decomposition (DMD) \cite{Schmid20105} and Operator Inference (OpInf) \cite{peherstorfer2016data} learn such best-fit operators from data via least-squares regression. Conversely, SINDy \cite{sindy} fits a nonlinear right-hand-side operator over a dictionary of candidate functions via sparse regression. However, in many applications the large dimensionality of $~u$ renders such standard approaches impractical due to high computational and storage costs. 

In practice, one is often interested only in the evolution of a few quantities of interest (QoIs), obtained by projecting the state onto a suitably chosen low-dimensional subspace. In low-rank DMD \cite{jovanovic2012low} and OpInf formulations, the QoIs are taken as modal amplitudes associated with a low-rank subspace identified from data, often via Proper Orthogonal Decomposition (POD) \cite{berkooz1993proper, Benner2015survey}. However, neither approach models the closure error arising from mode truncation, leading to inaccurate predictions when the discarded modes exert a non-negligible influence on the reduced dynamics. 

From the perspective of the Mori-Zwanzig formalism~\cite{mori1965transport, zwanzig1973nonlinear}, projecting onto such a low-dimensional subspace generally induces additional non-Markovian (memory) and noise contributions to the reduced dynamics. These contributions are neglected in standard low-rank DMD and OpInf, which approximate the reduced dynamics purely through the action of a time-local linear operator, resulting in an inexact model with closure error. An exact closed evolution equation that makes the additional contributions explicit is provided by the Mori-Zwanzig (MZ) identity. Suppose the state $~u$ can be decomposed into \emph{resolved} variables $~\phi \in \mathbb{R}^d$ (QoIs), and \emph{unresolved} variables $\widetilde{~\phi} \in \mathbb{R}^{\widetilde d}$, where $d + \widetilde d = N$ and $d \ll \widetilde d$. Then the MZ identity is a Volterra integro-differential equation (VIDE)~\cite{Wazwaz2011}
describing the exact dynamics of $~\phi$:
\begin{equation}\label{eq:vide}
\dot{~\phi}(t)
= ~g(t) 
+ \underbrace{~R(t)~\phi(t)\vphantom{\int_0^t ~K(t,s)\phi(s)\,ds}}_{\text{Markovian}}
+ \underbrace{~B(t,0)\widetilde{~\phi}(0)
+ \int_0^t ~B(t,s)\widetilde{~g}(s)\,ds\vphantom{\int_0^t ~K(t,s)~\phi(s)\,ds}}_{\text{Noise}}
+ \underbrace{\int_0^t ~K(t,s)\phi(s)\,ds\vphantom{\int_0^t ~K(t,s) ~\phi(s)\,ds}}_{\text{Non-Markovian}},
\end{equation}
Here, the linear operator $~R(t)$ plays the role of the instantaneous (or Markovian) contribution that is typically learned in standard approaches, while the noise kernel $~B(t,s)$ encodes the influence of unresolved initial conditions and forcing $\widetilde{~g}(t)$, and the memory kernel $~K(t,s)$ encodes the history dependence of $~\phi$. Despite the terminology, the ``noise'' term is deterministic; it is traditionally called noise only because it cannot be determined from $~\phi$ alone. Since the dimension of $\widetilde{~\phi}$ is typically much larger than that of $~\phi$, \eqref{eq:vide} can be viewed as an exact reduced-order model of the original system. However, computing solutions of \eqref{eq:vide} is challenging, as the integral terms are often expensive to evaluate, requiring efficient approximate treatments.

In this work, we formulate inverse problems that approximate the operators in \eqref{eq:vide} on a prescribed time mesh from trajectory snapshots via least-squares minimization, evaluating the integral terms using numerical quadrature. In our analysis, we adopt the following assumptions:
\begin{itemize}
\item \textbf{Linearity.} The dynamics are governed by \eqref{eq:cont_dyn}.
\item \textbf{Ensemble data.} We observe multiple trajectories generated from distinct initial conditions.
\item \textbf{Time-resolved snapshots.} Data are sampled on a uniform grid with spacing \(\Delta t\), sufficiently small to resolve the dynamics; discretization effects are neglected.
\item \textbf{Exact observations.} Snapshots are free of measurement error, and the dynamics are deterministic (no stochastic forcing).
\end{itemize}
When  $~A$ is constant in time, the kernels $~B$ and $~K$, are time-translation invariant, that is $~B(t,s) = ~B(t-s)$ and $~K(t,s) = ~K(t-s)$. However, when $~A$ varies in time, this is only an approximation. We nevertheless adopt this approximation even for $~A(t)$, since it makes the inverse problems far more tractable by significantly reducing the number of unknowns: the kernels depend only on the time lag instead of separately on $t$ and $s$. Viewed from this perspective, the memory and noise contributions act as corrective terms augmenting a standard time-local ROM. We refer to this time-translation invariance property as the \textit{stationarity assumption} in the rest of this paper.

Given $~\Phi_n \in \mathbb{R}^{N_s \times d}$ and $\widetilde{~\Phi}_n \in \mathbb{R}^{N_s \times \widetilde d}$ containing the resolved and unresolved states at time $t_n$ for $N_s$ trajectories, we consider the problem of reconstructing $~R$, $~K$, and $~B$ when $~A$ is unknown. We study \textbf{two distinct types of inverse problems}:

\begin{enumerate}
\item \textbf{Reconstruction from full observation data}: full observations of \(~u(t)\) are available for \(0 \le t_n \le T\), i.e., snapshots \(~\Phi_n\) and \(\widetilde{~\Phi}_n\) at each \(t_n\).
\item \textbf{Reconstruction from partial observation data}: only partial observations of \(~u(t)\) are available, i.e., snapshots \(~\Phi_n\) for \(0 \le t_n \le T\) together with the initial conditions \(\widetilde{~\Phi}_0\).
\end{enumerate}

{\bf Full Data} problems typically arise when constructing surrogate models from high-fidelity simulations, where full-state snapshot data are available. In contrast, {\bf Partial Data} problems are common when data come from experiments or from very large-scale simulations, where only the resolved component can be measured (or stored), making identification of the unresolved operator $\widetilde{~R}$ practically infeasible.

We propose a greedy, sequential-in-time reconstruction scheme to efficiently solve the global-in-time inverse problems. Within this framework, we show that {Full Data} problems are well-posed, while {Partial Data} problems are well-posed only when $~A(t)$ is time-invariant. For a time-varying $~A$, the partial data problem becomes ill-posed.

\begin{figure}
    \centering
    \includegraphics[width=\linewidth]{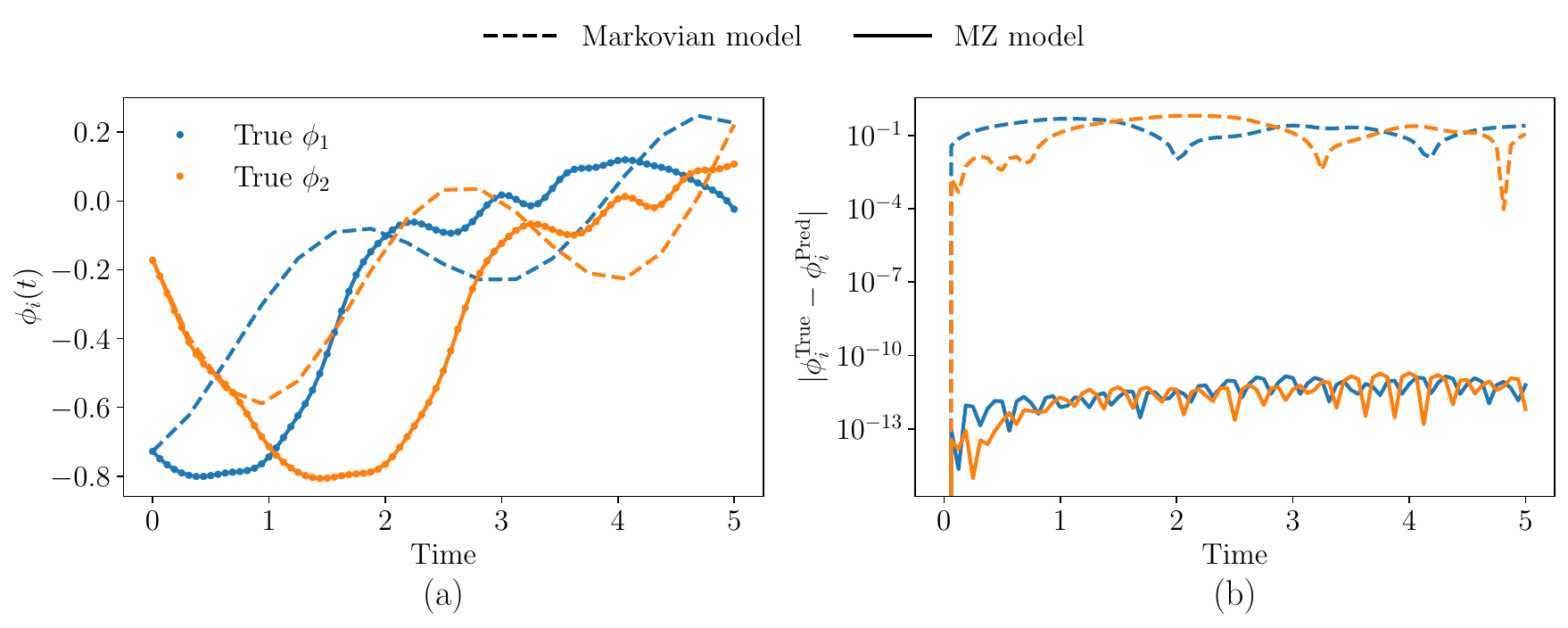}
    \caption{(a) True and predicted trajectories for two degrees of freedom (selected as the resolved variables) for a one-dimensional advection equation with periodic boundary conditions. Predictions are shown for a data-driven Markovian model and a data-driven MZ model \eqref{eq:vide}. (b) Absolute error in the predicted trajectories. The MZ model recovers the true trajectories, whereas the Markovian model exhibits large errors.}
    \label{fig:PODmodes}
\end{figure}

\subsection{Related work}
The Mori-Zwanzig (MZ) formalism has been used to derive or motivate closed reduced models in a wide range of applications, including coarse-grained molecular dynamics \cite{Lei2016, Li2015, Li2017a, Yoshimoto2013}, fluid and turbulent flows \cite{Parish2017a, Parish2017, Ma2019, Gouasmi2017, Wang2020, Maulik2020, tian2021data}, climate dynamics \cite{Kondrashov2015}, and uncertainty quantification \cite{Li2019,Venturi2017}. Direct approximation strategies for the MZ memory kernel are often developed within the Generalized Langevin Equation (GLE) framework, leading to short-memory approximations in the optimal prediction setting \cite{Chorin2002, Chorin2006a}, long-memory time-series expansions \cite{stinis2007}, and operator-series expansions \cite{Zhu2018, Zhu2018a}. Conversely, data-driven approaches seek to recover the memory integral or an equivalent non-Markovian closure contribution directly from data snapshots. Notable techniques include conditional-expectation estimators \cite{Brennan2018}, hierarchical rational parametrization in the Laplace domain \cite{Lei2016}, sparse polynomial regression \cite{Pan2018}, \emph{a priori} estimates via a pseudo-orthogonal dynamics equation in Liouville form \cite{Gouasmi2017}, and the NARMAX framework for stochastic parametrization \cite{Chorin2015, Lu2017a}. More recently, neural-network-based approximations of the memory integral (or equivalent non-Markovian closure) have been proposed, including LSTM-based models \cite{Wang2020, Maulik2020, gupta2025mori}, the Memory Neural Operators \cite{ruiz2025}, and deep neural networks that infer generalized Langevin memory kernels directly from data \cite{kerr2023deep}. In such methods, the data-driven memory approximation is coupled with additional architectures, such as POD-Galerkin models, neural ordinary differential equations, or autoencoders, to obtain closed models for the latent dynamics.

\subsection{Contributions}
While previous works have primarily emphasized application-specific modeling, problem formulation, and numerical experimentation, the numerical analysis of the inverse problem underlying the reconstruction of the memory kernel has received comparatively little attention. The principal objective of the present work is to address this gap by formulating a generic data-driven framework for learning MZ-based reduced models of driven linear dynamical systems and by investigating assumptions on the dynamics and data availability under which the MZ operators can be uniquely identified from discrete snapshot data. Although we restrict attention to linear systems, we anticipate that the analysis developed here will inform the design of data-driven memory closures for nonlinear systems in future work.

\noindent The main results of our work are summarized below:
\begin{itemize}
    \item For time-invariant $~A$, we show that the operators $~R$, $~B$, and
    $~K$ are uniquely recoverable within the proposed discrete reconstruction
    framework from both full and partial observation data; see
    \Cref{thm:type1,thm:type2}.
    \item For time-varying $~A(t)$, we show that the operators remain uniquely
    recoverable from full observation data, subject to the conditioning of the
    associated linear systems; see \Cref{thm:type1,remark:illcond}, and we
    identify additional assumptions on the forcing function that mitigate ill-conditioning; see
    \Cref{remark:fullrank}.
    \item For time-varying $~A(t)$ with partial observation data, we show that the greedy
    sequential reconstruction scheme is
    ill-posed; see \Cref{thm:type2}. To address this, we introduce a time-smoothing regularization that stabilizes the reconstruction.
\end{itemize}

\subsection{Outline}

The rest of the paper is organized as follows. \Cref{s:formulation} derives the MZ identity \eqref{eq:vide} for linear dynamical systems and discusses continuous and discrete formulations of both inverse problems. \Cref{s:methods} presents the core methodology, including reconstruction algorithms and our main well-posedness results. It also discusses learning and prediction under a finite-memory approximation and briefly analyzes the possibility of reconstructing non-stationary kernels. \Cref{s:results} applies the framework to a one-dimensional reaction-diffusion-advection equation with periodic boundary conditions, showing that it reliably reconstructs the MZ operators and accurately predicts the resolved dynamics. Finally, \Cref{app:proof} provides a detailed proof of \Cref{thm:type1}, \Cref{app:altdisc} presents alternative discretizations of the inverse problems, \Cref{app:pseudo} describes the PDE discretization, and \Cref{app:wave} presents the results of additional numerical experiments.

\section{Formulation}\label{s:formulation}
\subsection{Mori-Zwanzig formalism}

\begin{table}
  \centering
  \caption{Acronyms and important notations.}
  \begin{tabular}{ll}
    \toprule
    Acronym/symbol & Description \\
    \midrule
    MZ & Mori-Zwanzig \\
    QoIs & Quantities of interest \\
    ROM & Reduced Order Model \\
    VIDE & Volterra integro-differential equation\\
     $~u$ & High-dimensional state vector \\
     $N$ & Dimension of $~u$ \\
     $N_s$ & Number of trajectories \\
      $N_T$ & Number of points in time (assumed the same across trajectories)\\
     $~\phi$, $\widetilde{~\phi}$ & Vectors of resolved and unresolved variables \\
     $~g$, $\widetilde{~g}$ & Forcing vectors for resolved and unresolved variables \\
     $d, \widetilde d$ & Dimensions of $~\phi$ and $\widetilde{~\phi}$ \\
     $~R, \widetilde{~R}, ~K, ~B$ & Operators in the VIDE \eqref{eq:vide} reconstructed using inverse problems \\
      $~F_0$ & $N_s \times N$ matrix of initial conditions\\
      $~\Phi_n$ & $N_s \times d$ snapshot matrix of the resolved variables at time $t_n$ \\
      $\widetilde{~\Phi}_n$ & $N_s \times \widetilde d$ snapshot matrix of the unresolved variables at time $t_n$ \\
      $~G_n$ & $N_s \times d$ forcing matrix for the resolved variables at time $t_n$ \\
      $\widetilde{~G_n}$ & $N_s \times \widetilde d$ forcing matrix for the unresolved variables at time $t_n$ \\
      $~F^R_{n+1}, ~F^{KB}, ~F^{\text{reg}}_n$ &Least-squares operators in \eqref{eq:Rn}, \eqref{eq:R}, \eqref{eq:KB}, and \eqref{eq:reg_solve_2}\\
      $~Z^R_{n+1}, ~F^{KB}, ~Z^{\text{reg}}_n$ & Known matrices in \eqref{eq:Rn}, \eqref{eq:R}, \eqref{eq:KB}, and \eqref{eq:reg_solve_2}
  \end{tabular}
\end{table}

We illustrate the application of the MZ formalism to obtain a closed reduced model of the linear dynamical system \eqref{eq:cont_dyn}. The first step in the MZ procedure is to decompose the dynamics into resolved and unresolved components. Defining orthogonal projection operators $~P$ and $\widetilde{~P}$ by
\begin{align}\label{eq:proj_op}
    ~P~u := (~\phi, ~0), 
    \quad \widetilde{~P}~u := (~0, \widetilde{~\phi}), \quad \text{with} \quad 
    ~P + \widetilde{~P} = ~I, \quad
    ~P^2 = ~P, \quad
    \widetilde{~P}^2 = \widetilde{~P},
\end{align}
and projecting yields
\begin{subequations}\label{eq:mzdyn}
\begin{align}\label{eq:mzdyn1}
\frac{d}{dt}{~P~u} &= ~P\dot{~u} = ~P~A(~P + \widetilde{~P})~u + ~P\bar{~g}
= \underbrace{(~P~A~P)}_{\operatorname{Diag}(~R,~0)}\,~P~u
+ \underbrace{(~P~A\widetilde{~P})}_{\operatorname{Diag}(\widetilde{~R},~0)}\,\widetilde{~P}~u
+ \underbrace{~P\bar{~g}}_{~g}, \\
\frac{d}{dt}{\widetilde{~P}~u} &= \widetilde{~P}\dot{~u} = \widetilde{~P}~A(~P + \widetilde{~P})~u + {\widetilde{~P}\bar{~g}}
= \underbrace{(\widetilde{~P}~A~P)}_{\operatorname{Diag}(~0,~U)}\,~P~u
+ \underbrace{(\widetilde{~P}~A\widetilde{~P})}_{\operatorname{Diag}(~0,\widetilde{~U})}\,\widetilde{~P}~u
+ \underbrace{\widetilde{~P}\bar{~g}}_{\widetilde{~g}}. \label{eq:mzdyn2}
\end{align}
\end{subequations}
Let $~E(t,s)$ solve $\partial_t~E(t,s)=\widetilde{~P}~A(t)\widetilde{~P}\,~E(t,s)$ with $~E(s,s)=~I$. Then, integrating \eqref{eq:mzdyn2}, substituting the result in \eqref{eq:mzdyn1}, and rearranging yields an exact closed equation for the resolved components:
\begin{align}\label{eq:MZproj}
    ~P\dot{~u} 
    = ~P \bar{~g} + (~P~A~P) ~P~u  
    + {(~P~A\widetilde{~P})~E(t,0)}\widetilde{~P}~u_0 
    &+ \int_0^t {(~P~A\widetilde{~P}) ~E(t,s)} \widetilde{~P} \bar{~g}(s) \; ds \nonumber \\
    &+ \int_0^t {(~P~A\widetilde{~P}) ~E(t,s) (\widetilde{~P}~A~P)} ~P~u(s) \; ds.
\end{align}

\noindent If the full-state operator $~A$ is constant in time, the propagator has the simple form
$
    ~E(t,s) = e^{\widetilde{~P} ~A \widetilde{~P}(t-s)},
$
and the resolved dynamics evolve according to
\begin{align}
    \dot{~\phi} (t) 
    = ~g(t) + ~R~\phi(t)  
    + \underbrace{\widetilde{~R}e^{\widetilde{~U}t}\vphantom{\int_0^t}}_{~B(t)}\widetilde{~\phi}_0 
    + \int_0^t \underbrace{\widetilde{~R} e^{\widetilde{~U}(t-s)}\vphantom{\int_0^t}}_{~B(t-s)} \widetilde{~g}(s) \, ds
    + \int_0^t \underbrace{\widetilde{~R} e^{\widetilde{~U}(t-s)} ~U\vphantom{\int_0^t}}_{~K(t-s)} ~\phi(s) \, ds.
\end{align}
which has stationary memory and noise kernels. However, for a general time-varying $~A(t)$, the kernels $~K(t,s)$ and $~B(t,s)$ are not stationary, and simple convolution forms $~K(t-s)$ and $~B(t-s)$ are only modeling approximations. An application of the MZ framework to nonlinear dynamical systems is given in \cite{Wang2020}.

\subsection{Inverse problems}
We have demonstrated how the Mori-Zwanzig formalism can be applied to linear dynamical systems to obtain an exact dynamical equation for the resolved variables. In this subsection, we formulate inverse problems aimed at inferring the Markovian operator $~R(t)$, noise kernel $~B(t-s)$, and memory kernel $~K(t-s)$ directly from snapshot data. We require that the snapshot data be generated for $N_s$ different initial conditions and use the superscript notation $~\phi^{(j)}(t)$ and $\widetilde{~\phi}^{(j)}(t)$ to denote the resolved and unresolved states at time $t$ for the $j$-th trajectory.

\subsubsection{Formulation}
Given trajectories $\{\big(~\phi^{(j)}(t),\, \widetilde{~\phi}^{(j)}(t)\big)\}_{j=1}^{N_s}$ corresponding to $N_s$ independent initial conditions \linebreak $\{\big(~\phi^{(j)}_0, \, \widetilde{~\phi}^{(j)}_0\big)\}_{j=1}^{N_s}$, and forcings $\{\big(~g^{(j)}(t), \, \widetilde{~g}^{(j)}(t)\big)\}_{j=1}^{N_s}$, we construct two types of inverse problems presented in the definitions that follow.

\begin{defn}{(Continuous Full Data inverse problem)}\label{type:inv1}
Given full observation data \linebreak $\{\big(~\phi^{(j)}(t), \widetilde{~\phi}^{(j)}(t)\big)\}_{j=1}^{N_s}$ and forcings $\{\big(~g^{(j)}(t), \, \widetilde{~g}^{(j)}(t)\big)\}_{j=1}^{N_s}$, the Markovian operator $~R(t)$ and $\widetilde{~R}$ can be determined by solving
\begin{subequations}
\begin{align}
\arg\min_{~R, \widetilde{~R}} \; \frac{1}{2N_s} \sum_{j=1}^{N_s} \int_{0}^{T} 
\left\| \dot{~\phi}^{(j)}(t) - ~g^{(j)}(t) - ~R(t)~\phi^{(j)}(t) - \widetilde{~R}(t)\widetilde{~\phi}^{(j)}(t) \right\|_2^2 \, dt.
\end{align}    
Once $~R(t)$ is obtained, the memory kernel $~K(t-s)$ and the noise operator $~B(t)$ are determined by solving 
\begin{align}\label{type:inv1b}
    \arg\min_{~K, ~B} \; \frac{1}{2N_s} \sum_{j=1}^{N_s} \int_{0}^{T} 
\left\| 
    ~l^{(j)}(t) - ~B(t)\widetilde{~\phi}^{(j)}_0  - \int_{0}^{t} ~K(t-s)~\phi^{(j)}(s)\,ds  - \int_{0}^{t} ~B(t-s)\widetilde{~g}^{(j)}(s)\,ds 
\right\|_2^2 \, dt,
\end{align}
where $~l^{(j)}(t) = \dot{~\phi}^{(j)}(t) - ~g^{(j)}(t) - ~R(t)~\phi^{(j)}(t)$.
\end{subequations}
\end{defn}

\begin{defn}{(Continuous Partial Data inverse problem)}\label{type:inv2}
Given partial observation data \linebreak $\{\big(~\phi^{(j)}(t), \widetilde{~\phi}^{(j)}_0\big)\}_{j=1}^{N_s}$ and forcings $\{\big(~g^{(j)}(t), \, \widetilde{~g}^{(j)}(t)\big)\}_{j=1}^{N_s}$, the set of operators
$\{~R(t), ~K(t), ~B(t)\}$ is determined by solving
\begin{align}\label{eq:contpd}
    \underset{~R, ~K, ~B} {\operatorname{argmin}}\; \frac{1}{2N_s} \sum_{j=1}^{N_s} \int_{0}^{T} 
\left\| 
\dot{~\phi}^{(j)} (t)
- ~g^{(j)}(t)
- ~R(t)~\phi^{(j)} (t)
- ~B(t)\widetilde{~\phi}^{(j)}_0
- \int_{0}^{t} ~K(t-s)~\phi^{(j)}(s)\,ds 
\right.\nonumber  & \\ &
\left.
    \hspace{-40mm}
- \int_{0}^{t} ~B(t-s)\widetilde{~g}^{(j)}(s)\,ds
\right\|_2^2 \, dt .
\end{align}
\end{defn}

\subsubsection{Discrete formulation based on the Backward Euler rule}

We now present a discrete formulation of the inverse problems. For a simplified presentation, we present a first-order accurate time discretization that yields operators evaluated at time points consistent with the implicit Backward Euler forward solve \eqref{eq:time-march}. Alternative discretizations compatible with the Forward Euler and Implicit-Midpoint schemes---the latter being second-order accurate---are provided in \Cref{app:altdisc}. The theoretical results established in \Cref{s:methods} extend naturally to these alternative discretizations due to an analogous structure of the underlying least-squares problems.

Applying the Backward Euler rule to discretize the VIDE \eqref{eq:vide}, evaluating the $~K$-integral using a composite right-point quadrature and the $~B$-integral using a composite left-point quadrature, and rearranging terms, the following forward update can readily be obtained:
\begin{align}\label{eq:time-march}
    ~\phi_{n+1} = \left[~I-\Delta t \,~R_{n+1} - {\Delta t \,^2} ~K_0 \right]^{-1}
    \left[~\phi_n + \Delta t \, ~g_{n+1} 
    + \Delta t \, ~B_{n+1}\widetilde{~\phi}_0 
    + \Delta t \,^2\sum_{k=0}^{n-1}~K_{n-k} ~\phi_{k+1} 
    \right. & \nonumber \\ & \hspace{-30mm} \left. 
     + \Delta t \,^2\sum_{k=0}^{n}~B_{n-k+1} \widetilde{~g}_k 
    \right].
\end{align}
Here, $~I \in \mathbb R^{d \times d}$ denotes the identity matrix. The discrete inverse problems are naturally derived from this forward update as described below.

Given initial condition $(~\phi_0, \widetilde{~\phi}_0)$, we assume that snapshots $\{~\phi_n\}_{n=0}^{N_T}$ and $\{\widetilde{~\phi}_n\}_{n=0}^{N_T}$ can be sampled using an appropriate ODE time integrator at uniformly spaced $N_T$ time points $\{t_n\}_{n=0}^{N_T}$, where $~\phi_n = ~\phi(t_n)$. The data is stored in matrices $\{~{\Phi}_n\}_{n=0}^{N_T}$ and $\{\widetilde{~{\Phi}}_n\}_{n=0}^{N_T}$, where $~{\Phi}_n \in \mathbb{R}^{N_s \times d}$ and $\widetilde{~{\Phi}}_n \in \mathbb{R}^{N_s \times \widetilde d}$. Each row of $~{\Phi}_n$ (respectively, $\widetilde{~{\Phi}}_n$) corresponds to one trajectory, and each column corresponds to a resolved (respectively, unresolved) variable. Thus, $(~{\Phi}_n)_{ij}$ and $(\widetilde{~{\Phi}}_n)_{ij}$ denote the $j$-th resolved and unresolved variables at time $t_n$ for the $i$-th trajectory. The forcing matrices $\{~G_n\}_{n=0}^{N_T}$ and $\{\widetilde{~G}_n\}_{n=0}^{N_T}$ are defined analogously.

\begin{defn}{(Discrete Full Data inverse problem)}\label{def3}
Given full observation data $\{\big(~\Phi_n, \widetilde{~\Phi}_n\big)\}_{n=0}^{N_T}$ and forcing data $\{(~G_n, \widetilde{~G}_n)\}_{n=0}^{N_T}$, the set of Markovian operators $\{~R_{n+1}\}_{n=0}^{N_T-1}$ and $\{\widetilde{~R}_{n+1}\}_{n=0}^{N_T-1}$ can be reconstructed by solving
\begin{subequations}\label{eq:def3}
\begin{align}\label{eq:def3a}
 \underset{~R_{n+1}, \widetilde{~R}_{n+1}}{\operatorname{argmin}} \; 
 \frac{\Delta t}{2N_s} \sum_{n=0}^{N_T-1} 
\left\| 
    \dot{~\Phi}_{n+1} 
    - ~G_{n+1}
    - ~\Phi_{n+1}~R_{n+1}^\top 
    - \widetilde{~\Phi}_{n+1}\widetilde{~R}^\top_{n+1} 
\right\|_F^2,
\end{align}    
where $~R^\top$ denotes the transpose of $~R$. Matrices $\{~K_{n}\}_{n=0}^{N_T-1}$ and $\{~B_{n+1}\}_{n=0}^{N_T-1}$ can subsequently be reconstructed by solving 
\begin{align}\label{eq:def3b}
     \underset{~K_{n}, ~B_{n+1}}{\operatorname{argmin}}\; 
     \frac{\Delta t}{2N_s} \sum_{n=0}^{N_T-1}  
\left\| 
    ~L_{n+1}  
    - \widetilde{~\Phi}_{0}{~B}^\top_{n+1}
    - \Delta t \sum_{k=0}^{n}  ~\Phi_{k+1} ~K_{n-k}^\top 
    - \Delta t \sum_{k=0}^{n} \widetilde{~G}_{k} ~B_{n-k+1}^\top 
\right\|_F^2 \,,
\end{align}
where $~L_{n+1} = \dot{~\Phi}_{n+1} - ~G_{n+1} - ~\Phi_{n+1}~R_{n+1}^\top$.
\end{subequations}
\end{defn}

\begin{defn}{(Discrete Partial Data inverse problem)}\label{def4}
Given partial observation data $\{~\Phi_n\}_{n=1}^{N_T}$ and $\widetilde{~\Phi}_0$, and forcing data $\{(~G_n, \widetilde{~G}_n)\}_{n=0}^{N_T}$, matrices $\{~R_{n+1}, ~K_{n}, ~B_{n+1}\}_{n=0}^{N_T-1}$ can be reconstructed by solving
\begin{align}\label{eq:def4}
    \underset{~R_{n+1}, ~K_{n}, ~B_{n+1}} {\operatorname{argmin}}\; \frac{\Delta t}{2N_s} \sum_{n=0}^{N_T-1} 
\left\| 
\dot{~\Phi}_{n+1} - ~G_{n+1} - ~\Phi_{n+1}~R_{n+1}^\top
- \widetilde{~\Phi}_{0} {~B}^\top_{n+1}
\right. & \nonumber \\ & \hspace{-40mm} \left. 
- \Delta t \sum_{k=0}^{n} ~\Phi_{k+1} ~K_{n-k}^\top 
- \Delta t \sum_{k=0}^{n} \widetilde{~G}_{k} ~B_{n-k+1}^\top 
\right\|_F^2.
\end{align}
\end{defn}

Once the entire set $\{~R_{n+1}, ~K_{n}, ~B_{n+1}\}_{n=0}^{N_T-1}$ of operators has been reconstructed, the forward update \eqref{eq:time-march} defines a ROM that can be used to predict trajectories for unseen initial conditions.

\section{Reconstruction algorithms}\label{s:methods}


In this section we develop methodology to efficiently solve the discretized inverse problems described in \Cref{s:formulation}. In the Full Data problem, the minimization step \eqref{eq:def3a} for $\{(~R_n,\widetilde{~R}_n)\}$ decouples over $n$ into $N_T$ separate independent problems. However, the problem \eqref{eq:def3b} does not, since the contribution to the objective from time step $n$ depends on all past states $\{~K_j,~B_{j+1}\}_{j\le n}$ of the kernels. In principle, one could recover these kernels by solving a single global least-squares problem that couples all time levels, but this becomes computationally prohibitive as the number of snapshots grows. Instead, we exploit the causal, lower-triangular structure of the objective function and introduce a \emph{greedy time-marching least-squares scheme}, which solves only for the current pair $(~K_n,~B_{n+1})$ at each step while treating previously computed operators as fixed. 
Because past estimates are never revisited, this scheme only recovers an approximation to the true global minimizer. The time-marching technique is also used to solve the global minimization problem \eqref{eq:def4} in the partial data problem. Here, the unknowns at each time step are $(~R_{n+1}, ~K_n, ~B_{n+1})$.

We show that under mild assumptions, with the greedy time-marching scheme, the Full Data problem is well-posed. In contrast, the Partial Data problem is ill-posed unless the dynamics is governed by a time-invariant full-state operator $~A$. For ill-posed cases, we stabilize reconstruction by employing a time-smoothing Tikhonov regularization for $~R$ and $~K$. 

For the Full Data problem, we also examine reconstruction and prediction under a finite-memory approximation, in which the memory kernel $~K$ is assumed to vanish beyond a prescribed time lag and is thus compactly supported in time. Additionally, we briefly consider reconstruction of a fully non-stationary memory kernel $~K(t,s)$, without the stationarity assumption $~K(t,s) = ~K(t-s)$, and show that the resulting inverse problem is fundamentally ill-posed due to rank deficiency of the least-squares operator.

\subsection{Reconstruction with full observation data}\label{sec:type1}
\subsubsection[Reconstruction of R]{Reconstruction of $~R$}

Given full observation data, i.e., snapshots of both the resolved and
unresolved trajectories in time for multiple initial conditions, we first solve the minimization problem
\eqref{eq:def3a} to reconstruct the operators $~R_{n+1}$ and
$\widetilde{~R}_{n+1}$ at each time step. Defining
$~X_{n+1}^R := [~R_{n+1} \; \widetilde{~R}_{n+1}]^\top \in \mathbb R^{N \times d}$, where $N = d + \widetilde d$ and $N_s$ denotes the number of initial conditions, \eqref{eq:def3a} decouples across time and is equivalent to solving $N_T$ separate minimization problems
\begin{align}\label{eq:Rn}
 \underset{~X_{n+1}^R}{\operatorname{argmin}} \;  
\left\| ~Z_{n+1}^R - ~F_{n+1}^R ~X_{n+1}^R\right\|_F^2,
\qquad n=0,1,\ldots, N_T-1,
\end{align}
with the known matrix $~Z_{n+1}^R = \dot{~\Phi}_{n+1} - ~G_{n+1} \in \mathbb R^{N_s \times d}$
and least-squares operator $~F_{n+1}^R := [~\Phi_{n+1} \; \widetilde{~\Phi}_{n+1}] \in \mathbb R^{N_s \times N}$.
Thus, each time level $t_{n+1}$ corresponds to an independent least-squares fit. Assuming a dense factorization, the cost of factoring the least-squares operator is $\mathcal{O}(N_s N^2)$ flops per time step.

For time-invariant $~A$, the operators $~R$ and $\widetilde{~R}$ are also time-invariant. In this case, instead of $N_T$ separate problems, \eqref{eq:def3a} reduces to a single global least-squares problem for $~X^R := [~R \;\; \widetilde{~R}]^\top
\in \mathbb R^{N \times d}$:
\begin{align}\label{eq:R}
\underset{~X^R}{\operatorname{argmin}} \left\|~Z^R - ~F^R~X^R\right\|_F^2,
\end{align}
with the known matrix and least-squares operator
\begin{align*}
~Z^R := \begin{bmatrix}
    \dot{~\Phi}_{1} - ~G_{1} \\ \vdots \\ \dot{~\Phi}_{N_T} - ~G_{N_T}
\end{bmatrix} \in \mathbb R^{N_s N_T\times d},
\quad
~F^R 
:= \begin{bmatrix}
    ~\Phi_{1} & \widetilde{~\Phi}_{1} \\ \vdots  & \vdots \\ ~\Phi_{N_T} & \widetilde{~\Phi}_{N_T} 
\end{bmatrix}\in \mathbb R^{N_s N_T \times N},
\end{align*}
obtained by stacking the block matrices $\{~Z_{n+1}^R\}$ and
$\{~F_{n+1}^R\}$ row-wise. This global formulation uses all time levels simultaneously, improving conditioning. Factoring costs $\mathcal{O}(N_s N_T N^2)$ flops.

\subsubsection[Reconstruction of K and B]{Reconstruction of $~K$ and $~B$}

After reconstructing $~R$ at all prescribed time points, we next estimate the kernels $\{~K_n\}_{n=0}^{N_T-1}$ and $\{~B_{n+1}\}_{n=0}^{N_T-1}$ by solving the discrete problem \eqref{eq:def3b}. Defining 
\[
~L^{KB} :=
\begin{bmatrix}
~L_{1}\\[2pt]
\vdots\\[2pt]
~L_{N_T}
\end{bmatrix}
\in \mathbb{R}^{N_s N_T\times d},
\qquad
~X^K :=
\begin{bmatrix}
~K_0^\top\\[2pt]
\vdots\\[2pt]
~K_{N_T-1}^\top
\end{bmatrix}
\in \mathbb{R}^{d N_T\times d},
\qquad
~X^B :=
\begin{bmatrix}
~B_1^\top\\[2pt]
\vdots\\[2pt]
~B_{N_T}^\top
\end{bmatrix}
\in \mathbb{R}^{\widetilde d N_T\times d},
\]
\eqref{eq:def3b} can be written in the simplified form
\begin{align}
    \underset{~X_K, ~X_B}{\operatorname{argmin}}\;
\frac{\Delta t}{2N_s}\,\big\|~L^{KB} - ~F^K ({~X^K})^\top - ~F^B ({~X^B})^\top\big\|_F^2.
\end{align}
where both
\[
~F_K
=
\Delta t
\begin{bmatrix}
~\Phi_1           & ~0               & \cdots & ~0\\[2pt]
~\Phi_2           & ~\Phi_1          & \ddots & \vdots\\
\vdots           & \vdots          & \ddots & ~0\\[2pt]
~\Phi_{N_T}       & ~\Phi_{N_T-1}    & \cdots & ~\Phi_1
\end{bmatrix},
\qquad
~F_B
=
\begin{bmatrix}
\widetilde{~\Phi}_0 + \Delta t\,\widetilde{~G}_0 & ~0                            & \cdots & ~0\\[2pt]
\Delta t\,\widetilde{~G}_1                  & \widetilde{~\Phi}_0 + \Delta t\,\widetilde{~G}_0 & \ddots & \vdots\\
\vdots                                 & \vdots                       & \ddots & ~0\\[2pt]
\Delta t\,\widetilde{~G}_{N_T-1}            & \cdots                       & \Delta t\,\widetilde{~G}_1 & \widetilde{~\Phi}_0 + \Delta t\,\widetilde{~G}_0
\end{bmatrix},
\]
have a block lower-triangular Toeplitz structure, with each block being a non-square, tall matrix. The combined least-squares operator $[~F_K\ ~F_B]$ has size $N_s N_T \times N N_T$, where $N = d + \widetilde d$. Treating this as a dense least-squares problem would require factoring a matrix whose cost scales like $\mathcal{O}(N_s N^2 N_T^3)$ flops and $\mathcal{O}(N_s N N_T^2)$ storage. Since forming and solving this large global system is prohibitively expensive, we exploit the causal, lower-triangular structure of $~F_K$ and $~F_B$ and use a greedy but approximate time-marching approach. At time step $n$, we solve only for the unknowns $~K_{n}$ and $~B_{n+1}$ through the least-squares problem
\begin{align}\label{eq:triangularKB}
\underset{~K_{n},~B_{n+1}}{\operatorname{argmin}}\;
\frac{\Delta t }{2N_s}
\left\| 
    ~L_{n+1}  
    - (\widetilde{~\Phi}_{0} + \Delta t \,\widetilde{~G}_0 ){~B}^\top_{n+1}
    - \Delta t \sum_{k=0}^{n}  ~\Phi_{k+1} ~K_{n-k}^\top 
    - \Delta t \sum_{k=1}^{n}  \widetilde{~G}_{k} ~B_{n-k+1}^\top 
\right\|_F^2,
\end{align}
keeping the terms $\{~K_{n-k}\}_{k\ge 1}$ that have already been reconstructed at earlier time steps fixed. Similarly, in the second sum all terms depend only on previously computed $\{~B_{n-k+1}\}$, and for $n=0$ this sum is empty. Defining $~X^{KB}_{n+1} := [\Delta t ~K_n \; ~B_{n+1}]^\top \in \mathbb{R}^{N \times d}$, the above problem can be written as 
\begin{align}\label{eq:KB}
        \underset{~X_{n+1}^{KB}}{\operatorname{argmin}} ||{~Z}_{n+1}^{KB} - ~F^{KB} {~X}_{n+1}^{KB}||_F^2, \qquad n=0,1,\ldots, N_T-1,
\end{align}
with least-squares operator ${~F}^{KB} := [~\Phi_{1} \;\; (\widetilde{~\Phi}_0 + \Delta t \,\widetilde{~G}_0)]
\in \mathbb R^{N_s \times N}$ and known matrix
\[
{~Z}_{n+1}^{KB} := ~L_{n+1}
- \Delta t\sum_{k=1}^{n} ~\Phi_{k+1}~K_{n-k}^{\top}  
- \Delta t\sum_{k=1}^{n} \widetilde{~G}_{k} ~B_{n-k+1}^\top
\in \mathbb R^{N_s \times d}.
\]
The time-independent least-squares operator $~F^{KB}$ can be factored once and reused for all $n$, with a factoring cost of $\mathcal{O}(N_s N^2)$ flops. Because past estimates $\{~K_j,~B_{j+1}\}_{j<n}$ are never revisited, this greedy scheme only recovers an approximation to the true minimizer of the global problem

The next theorem establishes the well-posedness of the discrete Full Data inverse problem and provides bounds on the condition numbers of the least-squares operators $~F^R_n$, $~F^R$, and $~F^{KB}$ for the zero forcing case.

\begin{theorem}\label{thm:type1}
    The Full Data inverse problem is well-posed under the following assumptions:
    \begin{enumerate}[label= (\roman*)]
        \item $~F_0 := [~\Phi_0 \; \widetilde{~\Phi}_0]$ has full column rank.
        \item $\sigma_{\min}(\widetilde{~\Phi}_0 + \Delta t \, \widetilde{~G}_0) > 0$.
        \item $\operatorname{range}(~\Phi_0) \,\cap \, \operatorname{range}(\widetilde{~\Phi}_0 + \Delta t \,\widetilde{~G}_0) = \{0\}$.
    \end{enumerate}
    Moreover, if $~G(t)$ and $\widetilde{~G}(t)$ are zero matrices, the condition numbers of least-squares operators in \eqref{eq:Rn}, \eqref{eq:R}, and \eqref{eq:KB} satisfy
    \begin{gather}
    \kappa_2(~F_{n+1}^R) \leq \kappa_2(~F_0) e^{\int_0^{t_{n+1}} \lambda_{\max}(~H(s))-\lambda_{\min}(~H(s))ds}, \quad
    \kappa_2(~F^R) \leq \kappa_2(~F_0) \sqrt{\frac{\sum_{n=0}^{N_T-1} e^{2\int_0^{t_{n+1}}\lambda_{\max}(~H(s))ds}}{ \sum_{n=0}^{N_T-1} e^{2\int_0^{t_{n+1}}\lambda_{\min}(~H(s))ds}}}, \nonumber \\ 
    \kappa_2({~F^{KB}}) \leq \kappa_2(~F_0) \frac{\max\{e^{\int_0^{t_{1}}~\lambda_{\max}(~H(s)) ds},1\}}{\min\{e^{\int_0^{t_{1}}~\lambda_{\min}(~H(s)) ds},1\}},
    \end{gather}
    where $\lambda_{\max}(~H)$ and $\lambda_{\min}(~H)$ are the extremal eigenvalues of the Hermitian part $~H = \frac{~A + ~A^\star}{2}$ of the full-state linear operator $~A$ in \eqref{eq:cont_dyn}.
\end{theorem}

A detailed proof of the theorem is provided in \Cref{app:proof}. The key idea is to show that the least-squares operators $~F^R_{n+1}$ in \eqref{eq:Rn},  $~F^R$ in \eqref{eq:R}, and $~F^{KB}$ in \eqref{eq:KB} have full column rank under assumptions (i)--(iii), so the corresponding least-squares problems admit unique minimizers. Stability under small perturbations of the least-squares operator and known matrix then follows from standard perturbation results for overdetermined least-squares problems. The 2-norm condition number bounds are obtained by applying Grönwall's inequality to the full-state system and using the resulting growth estimates to bound the extremal singular values of the least-squares operators.

\begin{remark}{(Numerical ill-conditioning of $~F_{n+1}^R$)}\label{remark:illcond}
    Although \Cref{thm:type1} shows that the Full Data problem is well-posed, in many applications, the least-squares operator $~F_{n+1}^R$ may become numerically ill-conditioned as $n$ increases. To see this, consider a diffusion equation on $[0, 2\pi]$ with periodic boundary conditions discretized by the Fourier pseudospectral method with $N$ nodes. For the resulting dynamical system, 
    \[
        \lambda_{\min}(~H) = -\tfrac{N^2\mu}{4}, \qquad \lambda_{\max}(~H) = 0,
    \]
    where $\mu>0$ is the diffusion coefficient. The bound for $\kappa_2(~F_{n+1}^R)$ becomes
    \[
        \kappa_2(~F_{n+1}^R) \le \kappa_2(~F_0)\,e^{(\frac{\mu N^2}{4})t_{n+1}},
    \]
    which is tight since $~H$ and $~A$ share eigenvalues. The condition number thus grows exponentially in time with rate $\frac{\mu N^2}{4}$. In contrast, $~F^R$ remains well-conditioned, with the bound on its condition number given by:
    \[
        \kappa_2(~F^R) \leq \kappa_2(~F_0) \sqrt{N_T \frac{1-e^{-\frac{N^2\mu \Delta t}{2}}}{e^{-\frac{N^2\mu \Delta t}{2}}(1-e^{\frac{N^2\mu T}{2}})}} \sim \kappa_2(~F_0) \sqrt{\frac{N^2\mu T}{2}} \text{ as }  \Delta t \to 0.
    \]
 While this example assumes a constant diffusion coefficient, similar ill-conditioning can arise for a time-dependent $\mu(t)$, making the reconstruction of the time-dependent operator $~R(t)$ susceptible to numerical instability.
\end{remark}

\begin{remark}{(Full column rank requirement of $~G$)}\label{remark:fullrank}
    Although \Cref{thm:type1} does not explicitly require the forcing matrix $\bar{~G} = [~G \; \widetilde{~G}]$ to have full column rank, in problems like the one in the previous remark (with forcing added), this condition must hold to mitigate the ill-conditioning of $~F_{n+1}^R$ for large $n$. To see this, take $\bar{~G}$ to be constant. Then
    \[
        ~F_{n+1}^R 
        = e^{t_{n+1}~A^\top} ~F_0 
        + \int_0^{t_{n+1}} \bar{~G}\, e^{(t_{n+1}-s)~A^\top} ds
        = e^{t_{n+1}~A^\top} ~F_0 
        + \bar{~G}\, ~A^{-\top} (e^{t_{n+1}~A^\top} - ~I).
    \]
    As $n$ increases, $e^{t_{n+1}~A^\top} \to 0$ for a stable, invertible time-invariant operator $~A$, and thus $~F_{n+1}^R \sim -\,\bar{~G}\, ~A^{-\top}$, 
    and $\mathrm{rank}(~F_{n+1}) \leq \mathrm{rank}(\widetilde{~G})\;\mathrm{rank}(~A^{-\top})$.
\end{remark}

\subsubsection{Reconstruction under a finite-memory approximation}\label{sec:PH}

We now consider reconstruction under a finite-memory approximation, in which the memory kernel is assumed to vanish beyond a prescribed time lag so that the memory integrals have truncated support. For simplicity, we set $\bar {~g} = 0$. For the continuous problems \eqref{type:inv1b} and \eqref{eq:contpd}, this corresponds to replacing the first integral term by $\int_{\bar t_m}^t ~K(t-s) \phi^{(j)}(s) ds$ where $\bar t_m  = \max\{0, t-t_m\}$ and $t_m  > 0$. In the discrete setting, the finite-memory approximation leads to
\begin{align}\label{eq:KB_PH}
\underset{~K_{n},~B_{n+1}}{\operatorname{argmin}}\;
\frac{\Delta t }{2N_s} \sum_{n=0}^{N_T-1}
\left\| 
    ~L_{n+1}  
    - \widetilde{~\Phi}_{0} {~B}^\top_{n+1}
    -  \Delta t \sum_{k=\bar m}^{n}  ~\Phi_{k+1} ~K_{n-k}^\top 
\right\|_F^2,
\end{align}
where $\bar m = \max\{0, n-m\}$ with $m > 0$. Thus only the discrete lags $0,\dots,m$ of the memory kernel sequence $\{~K_n\}$ are retained, so $~K_n = 0$ for all $n>m$, and the discrete kernel has compact support in the time-lag index. The global unknowns are $\begin{bmatrix}
    ~B_1 & \cdots & ~B_{N_T} & \Delta t ~K_0 & \cdots \Delta t ~K_m
\end{bmatrix}^\top$, and the corresponding least-squares operator is $\begin{bmatrix}
    \operatorname{Diag}(~\Phi_0) & ~F^\text{trunc}
\end{bmatrix}$,
where 
\begin{align}\label{eq:block_trapezoidal}
    ~F^\text{trunc} = \begin{bmatrix}
        ~\Phi_1 \\
        ~\Phi_2 & ~\Phi_1 \\
        \vdots & \vdots  & \ddots \\
        ~\Phi_{m+1} & ~\Phi_m & \cdots & ~\Phi_1 \\
        \vdots & \cdots & \cdots & \vdots \\
        ~\Phi_{N_T} & ~ \Phi_{N_T-1} & \cdots & ~\Phi_{N_T-m+1} \\
    \end{bmatrix}
\end{align}
Due to the block-trapezoidal structure of \(~F^{\mathrm{trunc}}\), in contrast to the case for \eqref{eq:def3b}, a local least-squares solve over \(n\) neglects the final \(N_T-(m+1)\) rows, leading to biased or incomplete recovery. Accurate recovery thus requires a global solve of \eqref{eq:KB_PH}. Since forming \(~F^{\mathrm{trunc}}\) explicitly can be prohibitive in large-scale settings, we instead solve \eqref{eq:KB_PH} iteratively using a matrix-free implementation of the corresponding linear operator.

\begin{remark}{\textbf{(Reconstruction of a non-stationary memory kernel)}}
    In this remark, we briefly discuss solving \eqref{eq:def3b} after removing the stationarity assumption $~K(t,s) = ~K(t-s)$. For simplicity, we set $\bar {~g} = 0$ so that the second integral term disappears. With the fully non-stationary kernel $~K(t,s)$, \eqref{eq:def3b} changes to
\begin{subequations}
\begin{align}\label{eq:nonstat1}
     \underset{~K_{n}, ~B_{n+1}}{\operatorname{argmin}}\; 
     \frac{\Delta t}{2N_s} \sum_{n=0}^{N_T-1}  
\left\| 
    ~L_{n+1}  
    - \widetilde{~\Phi}_{0} {~B}^\top_{n+1}
    - \Delta t \sum_{k=0}^{n}  ~\Phi_{k+1} ~K_{nk}^\top 
\right\|_F^2 \,,
\end{align}
where $~K_{nk} = ~K(t_{n+1}, t_{k+1})$. Since the contributions to the objective are independent over $n$, \eqref{eq:nonstat1} decouples into $N_T$ separate least-squares problems:
\begin{align}\label{eq:nonstat2}
        \underset{~X_{n+1}}{\operatorname{argmin}} ||{~L}_{n+1} - ~F^{KB}_{n+1} {~X}_{n+1}||_F^2, \qquad n=0,1,\ldots, N_T-1,
\end{align}
where $~X_{n+1} = \begin{bmatrix}
    ~B_{n+1} & \Delta t ~K_{n+1,1} & \cdots & \Delta t ~K_{n+1,k+1} 
\end{bmatrix}^\top$ and $~F_{n+1}^{KB} = \begin{bmatrix}
    \widetilde{~\Phi}_0 & ~\Phi_1 & \cdots & ~\Phi_{n+1}
\end{bmatrix} $.
Since $\operatorname{range}(~\Phi_{n+1}) = \operatorname{range}(~\Phi_0)$ for all $n$, $~F_{n+1}^{KB}$ does not have full column rank for $n \ge 1$. Hence, the least-squares problem \eqref{eq:nonstat2} is not well-posed, since it does not admit a unique minimizer. 
\end{subequations}
\end{remark}

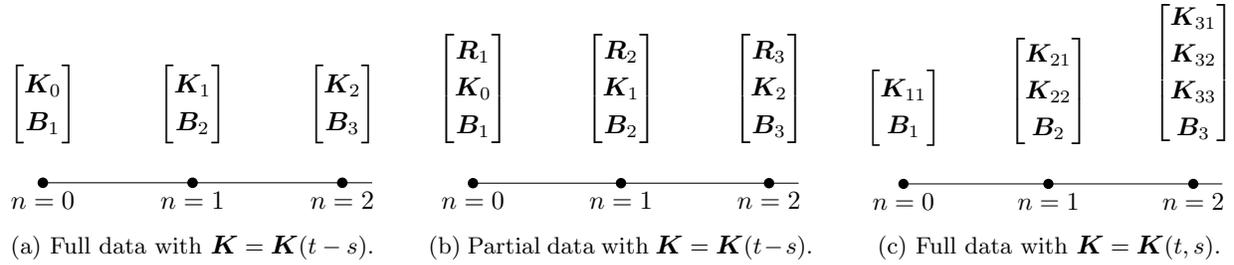
\begin{figure}[!h]
\centering

\begin{subfigure}[t]{0.31\textwidth}
\centering
\resizebox{\linewidth}{!}{%
\begin{tikzpicture}
  \draw (0,0) -- (4.4,0);
  \fill (0,0) circle (2pt) (2,0) circle (2pt) (4,0) circle (2pt);
  \node[below] at (0,0) {$n=0$};
  \node[below] at (2,0) {$n=1$};
  \node[below] at (4,0) {$n=2$};
  \node[above=4mm] at (0,0) {$\begin{bmatrix} ~K_0 \\ ~B_1 \end{bmatrix}$};
  \node[above=4mm] at (2,0) {$\begin{bmatrix} ~K_1 \\ ~B_2 \end{bmatrix}$};
  \node[above=4mm] at (4,0) {$\begin{bmatrix} ~K_2 \\ ~B_3 \end{bmatrix}$};
\end{tikzpicture}}
\caption{Full data with $~K = ~K(t-s)$.}
\end{subfigure}\hfill
\begin{subfigure}[t]{0.31\textwidth}
\centering
\resizebox{\linewidth}{!}{%
\begin{tikzpicture}
  \draw (0,0) -- (4.4,0);
  \fill (0,0) circle (2pt) (2,0) circle (2pt) (4,0) circle (2pt);
  \node[below] at (0,0) {$n=0$};
  \node[below] at (2,0) {$n=1$};
  \node[below] at (4,0) {$n=2$};
  \node[above=4mm] at (0,0) {$\begin{bmatrix} ~R_1\\ ~K_0 \\ ~B_1 \end{bmatrix}$};
  \node[above=4mm] at (2,0) {$\begin{bmatrix} ~R_2\\ ~K_1 \\ ~B_2 \end{bmatrix}$};
  \node[above=4mm] at (4,0) {$\begin{bmatrix} ~R_3\\ ~K_2 \\ ~B_3 \end{bmatrix}$};
\end{tikzpicture}}
\caption{Partial data with $~K = ~K(t-s)$.}
\end{subfigure}\hfill
\begin{subfigure}[t]{0.31\textwidth}
\centering
\resizebox{\linewidth}{!}{%
\begin{tikzpicture}
  \draw (0,0) -- (4.4,0);
  \fill (0,0) circle (2pt) (2,0) circle (2pt) (4,0) circle (2pt);
  \node[below] at (0,0) {$n=0$};
  \node[below] at (2,0) {$n=1$};
  \node[below] at (4,0) {$n=2$};
  \node[above=4mm] at (0,0) {$\begin{bmatrix} ~K_{11} \\ ~B_{1} \end{bmatrix}$};
  \node[above=4mm] at (2,0) {$\begin{bmatrix} ~K_{21} \\ ~K_{22} \\ ~B_{2} \end{bmatrix}$};
  \node[above=4mm] at (4,0) {$\begin{bmatrix} ~K_{31} \\ ~K_{32} \\ ~K_{33} \\ ~B_{3} \end{bmatrix}$};
\end{tikzpicture}}
\caption{Full data with $~K = ~K(t, s)$.}
\end{subfigure}
\caption{Unknowns to be reconstructed at the first three time steps in \eqref{eq:KB}, \eqref{eq:partialreform}, and \eqref{eq:nonstat2}, respectively.}
\end{figure}

\subsection{Reconstruction with partial data}\label{sec:type2}

For the case where only partial data are available, we use the Type 2 discrete inverse problem \eqref{eq:def4} to reconstruct the operators $\{(~R_{n+1},~K_n,~B_{n+1})\}_{n=0}^{N_T-1}$. As in the full-observation setting, we do not solve the globally coupled minimization \eqref{eq:def4} directly. Instead, we employ the greedy time-marching least-squares scheme introduced in \Cref{sec:type1}, updating only the current triple $(~R_{n+1},~K_n,~B_{n+1})$ at each time step while treating the previously reconstructed kernels $\{(~K_j,~B_{j+1})\}_{j<n}$ as fixed. At time step $n$, the minimization problem to be solve is
\begin{align}\label{eq:localpartial}
    \underset{~R_{n+1}, ~K_{n}, ~B_{n+1}} {\operatorname{argmin}}\; \frac{\Delta t}{2N_s}
\left\| 
\dot{~\Phi}_{n+1} - ~G_{n+1} - ~\Phi_{n+1}~R_{n+1}^\top
- (\widetilde{~\Phi}_{0} + \Delta t \,\widetilde{~G}_0 ){~B}^\top_{n+1}
\right. & \nonumber \\ & \hspace{-60mm}
\left. 
- \Delta t \sum_{k=0}^{n}  ~\Phi_{k+1} ~K_{n-k}^\top 
- \Delta t \sum_{k=1}^{n}  \widetilde{~G}_{k} ~B_{n-k+1}^\top 
\right\|_F^2.
\end{align}

\noindent With ${~X}_{n+1}^{RKB} := [~R_{n+1} \; \Delta t \, ~K_{n} \; ~B_{n+1}]^\top \in \mathbb R^{(2d+\widetilde d) \times d}$, we have the $N_T$ minimization problems
\begin{align}\label{eq:partialreform}
    \underset{{~X}_{n+1}}{\operatorname{argmin}}\; ||{~Z}_{n+1}^{RKB} - {~F}_{n+1}^{RKB} {~X}_{n+1}^{RKB} ||_F^2, \qquad n = 0, 1, \ldots, N_T -1,
\end{align}
where ${~Z}_{n+1}^{RKB} := \dot{~\Phi}_{n+1}- \Delta t \sum_{k=1}^{n} ~\Phi_{k+1}~K_{n-k}^{\top} - \Delta t \sum_{k=1}^{n} \widetilde{~G}_{k} ~B_{n-k+1}^\top \in \mathbb R^{N_s \times d}$, and  ${~F}_{n+1}^{RKB} := [~\Phi_{n+1}\; ~\Phi_{1} \; (\widetilde{~\Phi}_0+\Delta t \,\widetilde{~G}_0)] \in \mathbb R^{N_s \times (2d+\widetilde d)}$. 

Factoring $~F^{RKB}_{n+1}$ at each step costs $\mathcal{O}(N_s(2d+\widetilde d)^2)$ flops, so the overall cost of the sequential Type 2 reconstruction scales like $\mathcal{O}(N_T N_s (2d+\widetilde d)^2)$. 
Unlike in the full-observation setting, however, the problems \eqref{eq:partialreform}, in general, are not uniquely solvable. This is because the first two block columns of $~F^{RKB}_{n+1}$ are linearly dependent whenever all three operators $(~R_{n+1},~K_n,~B_{n+1})$ are independent unknowns. However, when $~R$ is treated constant, it can be absorbed into $~K_0$ by replacing this pair with the single operator $~R + \Delta t \,~K_0$. In this case, the block $~\Phi_1$ drops out of ${~F}_{n+1}^{RKB}$ and \eqref{eq:partialreform} becomes uniquely solvable. The following theorem makes this precise.

\begin{theorem}\label{thm:type2}
    Assume the conditions of \Cref{thm:type1} hold.
    \begin{enumerate}[label=(\roman*)]
        \item If $~R$ is constant in time, then, after combining $~R$ and
        $~K_0$ into the single operator $~R + \Delta t \,~K_0$ at $n=0$, each
        local least-squares problem \eqref{eq:partialreform} admits a unique
        minimizer.
        \item If $~R$ is time-varying, the local least-squares problems
        \eqref{eq:partialreform} do not admit unique minimizers.
    \end{enumerate}
\end{theorem}

\begin{proof}
    Since $\operatorname{range} (~\Phi_{n+1}) = \operatorname{range} (~\Phi_1)
    = \operatorname{range} (~\Phi_0)$ for all $n$, the least-squares operator \newline
    ${~F}_{n+1}^{RKB} = [~\Phi_{n+1}\;\; ~\Phi_{1} \;\; (\widetilde{~\Phi}_0+\Delta t \,\widetilde{~G}_0)]$
    fails to have full column rank whenever all three blocks are present.
    Thus, for time-varying $~R$, each local problem has infinitely many
    minimizers. If $~R$ is constant, the block $~\Phi_1$ drops out of
    ${~F}_{n+1}^{RKB}$ and, under assumptions (i)--(iii) of \Cref{thm:type1},
    the resulting least-squares operators have full column rank, so the local least-squares problems are uniquely solvable.
\end{proof}

\subsubsection{Regularized problem}
To address the ill-posedness of the sequential solve for the Type 2 discrete inverse problem, we augment the objective in \eqref{eq:localpartial} with time-smoothing Tikhonov penalties on $~R$ and $~K$, penalizing the differences $~R_{n+1} - ~R_n$ and $~K_n - ~K_{n-1}$. This choice stabilizes the reconstruction and yields a unique, regularized least-squares solution at each time step. Given parameters $\lambda_{~R}, \lambda_{~K} > 0$, the regularized problem at time step $n \ge 1$ reads
\begin{align}\label{eq:localpartialreg}
\underset{~R_{n+1}, ~K_{n}, ~B_{n+1}} {\operatorname{argmin}}\; 
\frac{\Delta t}{2N_s}
\left\| 
\dot{~\Phi}_{n+1} - ~G_{n+1} - ~\Phi_{n+1}~R_{n+1}^\top
- (\widetilde{~\Phi}_{0} + \Delta t \, \; \widetilde{~G}_0 ){~B}^\top_{n+1}
- \Delta t \sum_{k=0}^{n}  ~\Phi_{k+1} ~K_{n-k}^\top 
\right. & \nonumber \\ & \hspace{-130mm}
\left. 
- \Delta t \sum_{k=1}^{n} \widetilde{~G}_{k} ~B_{n-k+1}^\top 
\right\|_F^2
+ \frac{\lambda_{~R}}{2} \| ~R_{n+1} - ~R_{n} \|_F^2
+ \frac{\lambda_{~K}}{2} \| ~K_{n} - ~K_{n-1} \|_F^2,
\end{align}
assuming $~R_{n}$ and $~K_{n-1}$ have already been computed at step $n-1$. The problem for step $n=0$ can be combined with the problem for $n=1$ into a
single system, since the regularization intrinsically couples the pairs $(~R_1,~R_2)$ and $(~K_0,~K_1)$. This leads to the following sequence of least-squares problems:

\begin{itemize}
\begin{subequations}
\item For $n=0,1$, define 
$~X_1^\text{reg} = \begin{bmatrix}
    ~R_{1} & ~R_{2} & \Delta t \,~K_0 & \Delta t \,~K_1 &
    ~B_{1} & ~B_{2} 
    \end{bmatrix}^\top$
and solve
    \begin{align}\label{eq:reg_solve_1}
    \underset{~X_1^\text{reg}}{\operatorname{argmin}}\;
    \left\|~Z_1^\text{reg} - ~F_1^\text{reg}~X_1^\text{reg}\right\|_F^2,
    \end{align}    
with known matrix
$
~Z_1^\text{reg} := \begin{bmatrix}
    (\dot{~\Phi}_{1} - ~G_{1})^\top 
    & (\dot{~\Phi}_{2} - ~G_{2})^\top & ~0 & ~0
\end{bmatrix}^\top, 
$
and least-squares operator
\begin{gather*}
~F_1^\text{reg} := \begin{bmatrix}
    ~\Phi_{1} & 0 & ~\Phi_{1} & ~0 & \widetilde{~\Phi}_{0}+\Delta t \,\widetilde{~G}_0 & ~0\\
    ~0 & ~\Phi_{2} & ~\Phi_{2} & ~\Phi_{1} & \Delta t \,\widetilde{~G}_1 & \widetilde{~\Phi}_{0}+\Delta t \, \widetilde{~G}_0\\
    -\sqrt{\lambda_{~R}}\,~I & \sqrt{\lambda_{~R}}\,~I & ~0 & ~0 & ~0 & ~0\\
    ~0 & ~0 & -\sqrt{\lambda_{~K}}\,~I & \sqrt{\lambda_{~K}}\,~I & ~0 & ~0
\end{bmatrix}.
\end{gather*}
\item For $n > 1$, define
$~X_n^\text{reg} := \begin{bmatrix} ~R_{n+1}  & \Delta t \,~K_n & ~B_{n+1}\end{bmatrix}^\top$
and solve
\begin{align}\label{eq:reg_solve_2}
\underset{~X_n^\text{reg}}{\operatorname{argmin}}\;
\left\| ~Z_n^\text{reg} - ~F_n^\text{reg} ~X_n^\text{reg} \right\|_F^2,
\end{align}
where 
\begin{align*}
~Z_n^\text{reg} := \begin{bmatrix}
 ~J_{n+1}\\
\sqrt{\lambda_{~R}}~R_{n}\\
\sqrt{\lambda_{~K}}~K_{n-1}
\end{bmatrix},
\quad
~F_n^\text{reg} := \begin{bmatrix}
~\Phi_{n+1} & ~\Phi_{1} & \widetilde{~\Phi}_{0} + \Delta t \, \widetilde{~G}_0\\
\sqrt{\lambda_{~R}}\, ~I & ~0 & ~0\\
~0 & \sqrt{\lambda_{~K}}\, ~I & ~0
\end{bmatrix},
\end{align*}    
and $~J_{n+1} = \dot{~\Phi}_{n+1} - ~G_{n+1}
- \Delta t \sum_{k=1}^{n}  ~\Phi_{k+1} ~K_{n-k}^\top
- \Delta t \sum_{k=1}^{n}  \widetilde{~G}_{k} ~B_{n-k+1}^\top$.
\end{subequations}
\end{itemize}
If $~F_0$ has full column rank, the matrices ${~F_1^\text{reg}}^\top ~F_1^\text{reg}$ and
${~F_n^\text{reg}}^\top ~F_n^\text{reg}$ are positive definite for $\lambda_{~R}, \lambda_{~K} > 0$, and thus invertible. Hence, both \eqref{eq:reg_solve_1} and \eqref{eq:reg_solve_2} admit unique minimizers.

\section{Results}\label{s:results}

We evaluate the accuracy and performance of the proposed model reduction framework through computational experiments, reconstructing operators in \eqref{eq:vide} from synthetically generated data, and predicting resolved dynamics for unseen initial conditions. The experiments focus on linear dynamical systems obtained after spatially discretizing partial differential equations (PDEs). We choose a small subset of the degrees of freedom of the system as the resolved variables. The goal of the experiments is to accurately predict the trajectories of the resolved variables through the MZ model \eqref{eq:vide}, with operators reconstructed from full and partial observation data. 

We consider two parametric PDEs on a one-dimensional periodic domain: a reaction-diffusion-advection equation and a damped wave equation (discussed in \Cref{app:wave}), and generate multiple test cases by varying the PDE parameters. To ensure that the initial data matrix $~F_0 \in \mathbb{R}^{N_s \times N}$ has full column rank, we generate $N_s \ge N$ independent smooth initial conditions for training using the following procedure:
\begin{enumerate}
  \item For each $s = 1, \ldots, N_s$, set
  \[
  ~u^s = \sum_{m=0}^{N/2} \frac{a_m^s}{1+m^2} \cos(m ~x -  \phi_m^s),
  \qquad a_m^s \sim U(0,1), \quad \phi_m^s \sim U(0, 2\pi),
  \]
  where $~x \in \mathbb R^N$ is the vector of spatial grid points.
  \item Stack all $~u^s$ as rows of $~F_0$.
\end{enumerate}

Here, $N$ is even, and $U(a,b)$ denotes the uniform distribution on $[a,b]$. A low-pass filter is applied to dampen the high-frequency modes by dividing the amplitudes $a_m^s$ by $1+m^2$, which yields spatially smooth initial data. The testing data are generated independently using the same procedure. Once the matrices of initial conditions are obtained, the linear dynamical systems are integrated in time using an adaptive high-order Runge-Kutta scheme (the RK45 implementation in SciPy's \texttt{solve\_ivp}) with a prescribed solver tolerance $\tau_{\text{sol}} = 10^{-12}$ for all experiments. After appropriate projections, this yields the datasets $\{\Phi(t_n)\}_n$ and $\{\widetilde{\Phi}(t_n)\}_n$ needed for operator reconstructions. Operators are reconstructed in both full and partial observation data regimes, with results reported in the latter only for ill-posed cases. For all experiments, we employ the second-order accurate schemes of \Cref{sec:impmid} for reconstruction and subsequent trajectory prediction.

For reconstruction with full observation data, we verify the second-order convergence of the $~K$ and $~R$ operators on a sequence of temporal discretizations. The error metrics considered are 
\begin{itemize}
  \item Relative Frobenius error in $~K$ over $[0,T]$:
  \begin{align}\label{eq:errK}
       \mathcal E_{~K} = \sqrt{\frac{\int_0^T \|~K(t) - ~K^\text{true} (t)\|_F^2 \,dt}
    {\int_0^T \|~K^\text{true} (t)\|_F^2 \,dt}},
  \end{align}
  where $~K$ and $~K^{\text{true}}$ denote the reconstructed and ground-truth operators, respectively. Since an analytical form of $~K^\text{true}$ is not available when the PDE parameters are time-varying, we replace $~K^\text{true}$ by $~K$ reconstructed at the finest temporal resolution, yielding a self-convergence error metric.

    \item Relative Frobenius error in $~R$ over $[0,T]$:
  \begin{align}\label{eq:errR}
    \mathcal E_{~R} =
    \sqrt{\frac{\int_0^T \|~R(t) - ~R^\text{true}(t)\|_F^2 \,dt}
    {\int_0^T \|~R^\text{true}(t)\|_F^2 \,dt}}.
  \end{align}

  \item Relative Frobenius error in $~\Phi^\text{pred}$ over $[0,T]$:
\begin{align}\label{eq:errphi}
    \mathcal E_{~\Phi^\text{pred}} = \sqrt{\frac{\int_0^T \|~\Phi^\text{pred}(t) - ~\Phi(t)\|_F^2 \,dt}
  {\int_0^T \|~\Phi(t)\|_F^2 \,dt}},
\end{align}
where $~\Phi^\text{pred}$ contains the trajectories of resolved variables for test initial conditions obtained by time-integrating \eqref{eq:vide} after reconstructing the operators.
\end{itemize}
The time integrals in \eqref{eq:errK}-\eqref{eq:errphi} are approximated by the same numerical quadrature used to integrate the noise and memory terms in \eqref{eq:vide}. For the reconstruction of the memory kernel under a finite-memory approximation, we solve the least-squares problem \eqref{eq:KB_PH} matrix-free and iteratively using SciPy's \texttt{lsqr} solver. This solver employs an iterative Krylov subspace method based on the Golub-Kahan bidiagonalization process \cite{paige1982lsqr} and efficiently minimizes the residual and solution norms for large sparse least-squares problems.

\noindent Below, we present a brief summary of the results presented in the various tables and figures of this section:

\begin{itemize}

\item \Cref{tab:pde1} lists the eight test cases generated by varying the PDE parameters—the diffusion coefficient \(\mu(t)\), velocity \(v(t)\), reaction strength \(a(t)\), and forcing \(\bar g(t,x)\). Both time-varying and constant parameters are considered.

\item \Cref{tab:eq_1_conv} reports the reconstruction errors $\mathcal E_{~R}$ and $\mathcal E_{~K}$, showing second-order temporal convergence, together with the prediction errors $\mathcal E_{~\Phi^\text{pred}}$, for all the tested cases and various time step sizes, with operators reconstructed from full observation data. The corresponding true and predicted trajectories are shown in \Cref{fig:pde1}.

\item \Cref{tab:eq1_PH} reports the prediction errors for the first six test cases (no forcing) for varying choices of the finite-memory approximation parameter $m$, with operators reconstructed from full observation data. The time evolution of $\|~K\|_F$ and the predicted trajectories for different choices of $m$ are shown in \Cref{fig:pde_PH}.

\item \Cref{tab:eq1_reg} reports the prediction errors $\mathcal E_{~\Phi^\text{pred}}$, with operators reconstructed from partial observation data. Results are shown only for ill-posed cases for different combinations of the time-smoothing regularization parameters $\lambda_{~R}$ and $\lambda_{~K}$.
\end{itemize}

\begin{table}[h!]
\centering
\caption{Different cases and parameter sets for the reaction-diffusion-advection equation. $\mu(t)$ is the diffusion coefficient, $v(t)$ is velocity, $a(t)$ is reaction strength, and $\bar g(t, x)$ is prescribed forcing.}
\label{tab:pde1}
\begin{tabular}{|c|c|c|c|c|l|}
\hline
\textbf{Case} & {$\mu(t)$} & {$v(t)$} & {$a(t)$} & {$\bar{g}(t,x)$} & \textbf{Description}\\
\hline
\textbf{(a)} & 0 & 1 & 0 & 0 & Advection only\\
\hline
\textbf{(b)} & 0.1 & 1 & 0 & 0 & Advection dominated\\
& & & & & without reaction\\
\hline
\textbf{(c)} & 0.05 & 1 & 1 & 0 & Reaction-diffusion-advection\\
\hline
\textbf{(d)} & 2 & 0.5 & 0 & 0 & Diffusion dominated\\
& & & & & without reaction\\
\hline
\textbf{(e)} & $0.01\cos^{2}(2t)$ & $1 + \sin^{2}(5t)$ & $-0.5\cos t$ & 0 &
Time-dependent \\
& & & & & reaction-diffusion-advection\\
\hline
\textbf{(f)} & 0 & $0.5 + \sin(4t)$ & $\cos^{2}(5t)$ & 0 &
Time-dependent\\
& & & & & reaction-advection\\
\hline
\textbf{(g)} & 2 & 0 & 0 & 0.001 &
Constant diffusion\\
& & & & & with forcing\\
\hline
\textbf{(h)} & $2 + 0.25\sin t + 0.1\cos 10t$ & 0 & 0 &
$\dfrac{u_0(x)}{1+t}$ &
Time-dependent diffusion\\
& & & & & with parametric forcing\\
\hline
\end{tabular}
\end{table}

\subsection{Reaction-diffusion-advection equation}\label{sec:pde1}
We consider the reaction-diffusion-advection equation with periodic boundary conditions:
\begin{equation}\label{eq:pde1}
  \left \{ \begin{aligned}
    &\partial_t u (t,x) = v(t) \partial_x u (t, x) + \mu (t) \partial_{xx} u(t, x) + a(t) u (t, x) + \bar g(t, x), \quad && t \in [0, T], \quad && x \in [0, 2\pi], \\
    & u(0, x) = u_0(x), \quad && x \in [0, 2\pi], \\
    & u(t, 0) = u(t, 2\pi), \quad && t \in [0, T],
  \end{aligned} \right.
\end{equation}
where $\mu(t)$ is the diffusion coefficient, $v(t)$ is the advection velocity, $a(t)$ is the reaction strength, and $\bar g(t, x)$ is a $2\pi$-periodic prescribed forcing term. We discretize \eqref{eq:pde1} in space using the Fourier pseudospectral method
on the uniform grid $\{x_j = j\Delta x \}_{j=0}^{N-1}$ with
$\Delta x = 2\pi/N$, which yields the following semi-discrete system of ODEs:
\begin{equation}  \label{eq:ode1}
\frac{d~{u}}{dt} = \left[ {\mu}(t) ~{D}^2 + {v}(t) ~{D} + {a(t)}~I  \right] ~{u} + \bar{~g}(t).
\end{equation}
Here, $~{u} = [ u_0(t), \dots, u_{N-1}(t)]^{\top}$, $\bar {~g} = [\bar g(t, x_0), \dots, \bar g(t, x_{N-1})]^\top$, $~{D}$ denotes the spectral differential matrix, and $~I$ is the identity. The terminal time is $T = 5$. We set $N = 30$ and define the resolved variables as $~\phi = [u_5, u_{10}, u_{15}, u_{20}, u_{25}] \in \mathbb R^5$, with the remaining components of $~u$ forming $\widetilde{~\phi}$, obtaining operators $~P$ and $\widetilde{~P}$. The goal of the experiment is to verify that the reconstructed operators yield accurate trajectories of the resolved variables as the PDE parameters are varied. 

Varying parameters, we construct eight test cases summarized in \Cref{tab:pde1}. Cases (a)-(f) have zero forcing; among these, (a)-(d) use time-independent parameters, while (e)-(f) use time-varying parameters.
For reasons described in \Cref{remark:illcond}, the least-squares problem \eqref{eq:Rn} can become numerically ill-conditioned when $\mu(t)$ is time-dependent and large. In that regime, the rank of the design matrix $~F_{n+1}^{R}$ can effectively collapse in time unless a stabilizing forcing term is added. As explained in \Cref{remark:fullrank}, the sampled forcing matrix $\bar {~G} \in \mathbb R^{N_s \times N}$ must have full column rank.
Therefore, in case (h), which combines nonzero forcing with a large time-dependent diffusion coefficient, we construct $\bar g(x, t)$ as a function of the initial condition so that $\bar{~G}$ inherits full column rank from the sampled initial conditions. The only other case with nonzero forcing is (g); there, $v = a = 0$ and diffusion is constant; full column rank of $\bar {~G}$ is not required. 

We generate $N_s = 35$ independent initial conditions for training and 15 for testing. The solver \texttt{solve\_ivp} produces dense interpolants, which are used to assemble the training data matrices ${~\Phi_n}$ and ${\widetilde{~\Phi}_n}$ over several temporal discretizations. For each discretization, once $~R$, $~K$, and $~B$ are reconstructed with full observation data through \eqref{eq:def3}, trajectories for test initial conditions are predicted via \eqref{eq:implicit_midpoint}. The relative errors of \eqref{eq:errK}, \eqref{eq:errR}, and \eqref{eq:errphi} are evaluated and reported in \Cref{tab:eq_1_conv}. The results confirm second-order convergence of the reconstructed operators $~K$ and $~R$ and show that the errors in $~\Phi^\text{pred}$ are small and consistent with the solver tolerance $\tau_\text{sol}$ used to generate training data. For cases with time-varying parameters, a reported error of zero for $~K$ on the finest time grid indicates that the reconstructed operator at that $\Delta t$ was used in place of the analytically unavailable $~K^\text{true}$ to compute the error. The use of a stationary kernel does not lead to noticeable error in $~\Phi^\text{pred}$ for these cases.

\begin{table}
  \centering
  \caption{Relative errors in the predicted trajectories $~\Phi^\text{pred}$ and operators $~R$ and $~K$ for various $\Delta t$ across all eight reaction-diffusion-advection test cases. Both $~K$ and $~R$ exhibit second-order convergence and $~\Phi^\text{pred} \sim \tau_{sol}$. A zero error for $~K$ indicates it was used in \eqref{eq:errK} because $~K^\text{true}$ was unavailable. Operators were reconstructed with full data.}
  \label{tab:eq_1_conv}
  \resizebox{\textwidth}{!}{%
  \begin{tabular}{
    c
    S[table-format=1.4e-2]
    S[table-format=1.2e-03]
    S[table-format=1.2e-03]
    S[table-format=1.2]
    S[table-format=1.2e-03]
    S[table-format=1.2]
  }
    \toprule
    Case & {$\Delta t$} & {Error in $~\Phi^\text{pred}$} & {Error in $~K$} & {Rate} & {Error in $~R$} & {Rate}\\
    \midrule
    (a) & 3.1250e-02  & 1.78e-11 & 4.57e-02 & {---} & 1.82e-02 & {---} \\
        & 1.5625e-02  & 1.78e-11 & 1.13e-02 & 2.01 & 4.49e-03 & 2.02 \\
        & 7.8125e-03  & 1.78e-11 & 2.83e-03 & 2.00 & 1.12e-03 & 2.01 \\
    \midrule
    (b) & 3.1250e-02  & 1.70e-12 & 1.85e-02 & {---} & 2.11e-02 & {---} \\
        & 1.5625e-02  & 1.71e-12 & 4.50e-03 & 2.04 & 5.44e-03 & 1.95 \\
        & 7.8125e-03  & 1.69e-12 & 1.10e-03 & 2.03 & 1.37e-03 & 1.99 \\
    \midrule
    (c) & 3.1250e-02  & 1.95e-12 & 6.56e-03 & {---} & 3.06e-02 & {---} \\
        & 1.5625e-02  & 1.96e-12 & 1.64e-03 & 2.00 & 7.59e-03 & 2.01 \\
        & 7.8125e-03  & 1.96e-12 & 4.10e-04 & 2.00 & 1.89e-03 & 2.00 \\
    \midrule
    (d) & 1.9531e-03  & 4.62e-12 & 1.71e-02 & {---} & 2.39e-02 & {---} \\
        & 9.7656e-04  & 4.63e-12 & 4.64e-03 & 1.88 & 6.18e-03 & 1.95 \\
        & 4.8828e-04  & 4.63e-12 & 1.20e-03 & 1.95 & 1.56e-03 & 1.99 \\
    \midrule
    (e) & 1.5625e-02  & 1.13e-11 & 1.52e-02 & {---} & 1.39e-02 & {---} \\
        & 7.8125e-03  & 1.13e-11 & 3.73e-03 & 2.03 & 3.43e-03 & 2.02 \\
        & 3.9063e-03  & 1.13e-11 & 7.64e-04 & 2.29  & 8.56e-04 & 2.00  \\
        & 1.9531e-03  & 1.13e-11 & 0.00e+00 & {---} & 2.14e-04 & 2.00 \\
    \midrule
    (f) & 1.5625e-02  & 1.36e-11 & 2.93e-02 & {---} & 6.82e-03 & {---} \\
        & 7.8125e-03  & 1.36e-11 & 7.14e-03 & 2.03  & 1.69e-03 & 2.01 \\
        & 3.9063e-03  & 1.35e-11 & 1.44e-03 & 2.30  & 4.22e-04 & 2.00 \\
        & 1.9531e-03  & 1.35e-11 & 0.00e+00 & {---} & 1.05e-04 & 2.00 \\
    \midrule
    (g) & 1.9531e-03  & 4.18e-12 & 1.71e-02 & {---} & 2.40e-02 & {---} \\
        & 9.7656e-04  & 4.21e-12 & 4.64e-03 & 1.88  & 6.19e-03 & 1.95 \\
        & 4.8828e-04  & 4.45e-12 & 1.20e-03 & 1.95  & 1.56e-03 & 1.99 \\
    \midrule
    (h) & 3.9063e-03  & 2.67e-12 & 3.63e-01 & {---} & 5.18e-03 & {---} \\
        & 1.9531e-03  & 3.14e-12 & 1.34e-01 & 1.44  & 1.44e-03 & 1.85 \\
        & 9.7656e-04  & 9.08e-12 & 3.39e-02 & 1.98  & 3.71e-04 & 1.96  \\
        & 4.8828e-04  & 1.80e-11 & 0.00e+00 & {---} & 9.34e-05 & 1.99 \\
  \end{tabular}
  }
\end{table}

\Cref{fig:pde1} presents the trajectories of the resolved variables for a single test initial condition and all eight cases. The predicted trajectories, computed through \eqref{eq:vide}, with reconstructed operators, are shown as solid colored curves, while the true trajectories, computed through \eqref{eq:ode1} and subsequently projected, are shown as dashed black curves. The plots demonstrate close agreement between the predicted and true trajectories.

Next, we consider cases (a)-(e) and study prediction under a finite-memory approximation of the memory kernel $~K$, obtained by solving the least-squares problem \eqref{eq:KB_PH} after having computed $~R$. The time step is fixed at $\Delta t = 0.03125$, and the finite-memory approximation parameter $m$ is set sequentially to $0.1N_T$, $0.2N_T$, $0.3N_T$, $0.5N_T$, and $N_T$. For a given $m$, only contributions from time lags up to $t_m = m \Delta t$ are retained by setting $~K_n = 0$ for all indices $n > m$, so that the discrete kernel sequence $\{~K_n\}_{n=0}^{N_T}$ is effectively compactly supported on $\{0,\dots,m\}$, where the support varies from $10\%$ (shortest memory) to $100\%$ (full memory).

We perform 500 iterations of \texttt{lsqr} and report the error in $~\Phi^{\text{pred}}$ for each value of $m$. The results in \Cref{tab:eq1_PH} show that, for cases (b), (d), and (e)---all of which have nonzero diffusion coefficients--the finite-memory approximation yields reasonably small errors. In particular, for case (d) the error is below $0.1\%$ even with only $10\%$ memory support. For the other two cases, the error is approximately $2\%$ at $10\%$ memory support and decreases as the support increases. In contrast, for cases (c) and (f), which have a large reaction coefficient and small or no diffusion, truncating the memory support leads to large errors; for case (f), the error exceeds $100\%$ for all $m < N_T$. In some cases, the errors are not comparable to $\tau_\text{sol}$ even with full memory because 500 iterations are insufficient for the optimization error to vanish. Nevertheless, they remain small enough for the purposes of this experiment.

\Cref{fig:pde_PH} shows the time evolution of the norm of the reconstructed kernel $~K$ together with the trajectory of the first resolved variable for all values of $m$. In cases (b), (d), and (e), $\lVert~K\rVert_F$ decays to zero in time, and correspondingly the trajectories remain reasonably accurate even with short memory. In contrast, for $m < 0.5N_T$, the trajectories in cases (a) and (c) deviate noticeably from the true trajectories due to the non-decaying kernels. In case (f), large errors arise for all values of $m$, as $\lVert~K\rVert_F$ grows in time to a large magnitude. 

Finally, we examine reconstruction from partial data, focusing on cases (e) and (f), for which the underlying inverse problem is ill-posed due to time-varying PDE parameters. We recover $~R$, $~K$, and $~B$ using the regularized formulation \eqref{eq:localpartialreg} for different choices of the regularization parameters. Specifically, we set $\Delta t = 0.015625$ and record the relative error in $~\Phi^\text{pred}$ for parameter pairs $(\lambda_{~R}, \lambda_{~K})$, varying each parameter over the set $\{0, 10^{-8}, 10^{-4}, 10^{-2}, 1\}$.  The results, shown in \Cref{tab:eq1_reg}, reveal that operators reconstructed without regularization do not yield accurate solutions. When $\lambda_{~R} = \lambda_{~K} = 10^{-8}$, the relative error in the solution is small and consistent with $\tau_\text{sol}$. The errors are larger but still reasonable (on the order of $10^{-6}$) when $\lambda_{~K}=0$ and $\lambda_{~R} > 0$. This is because the regularized problem \eqref{eq:localpartialreg} does not strictly require $\lambda_{~K} > 0$ to admit a unique minimizer. Increasing $\lambda_{~K}$ above $10^{-8}$ degrades the accuracy, with errors exceeding $100\%$ unless $\lambda_{~R} \ge \lambda_{~K}$.

\begin{figure}
    \centering
    \includegraphics[width=\linewidth]{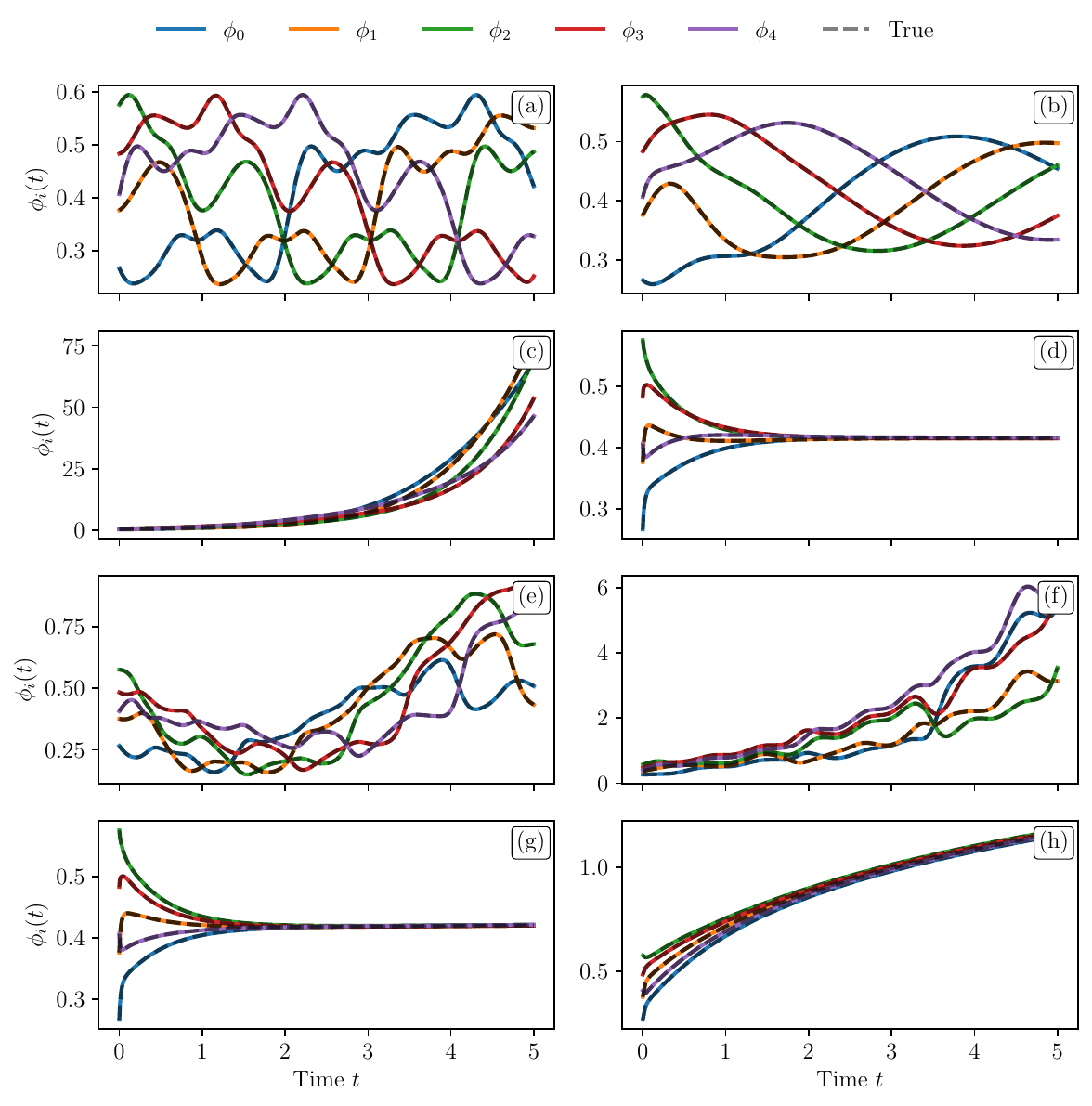}
    \caption{Predicted trajectories of the resolved variables $~\phi$ in all the reaction-diffusion-advection test cases for a single test initial condition. The trajectories are obtained by solving \eqref{eq:vide}, with operators $~R$, $~K$, and $~B$ reconstructed using full observation data.}
    \label{fig:pde1}
\end{figure}

\begin{figure}
    \centering
    \includegraphics[width=0.85\linewidth]{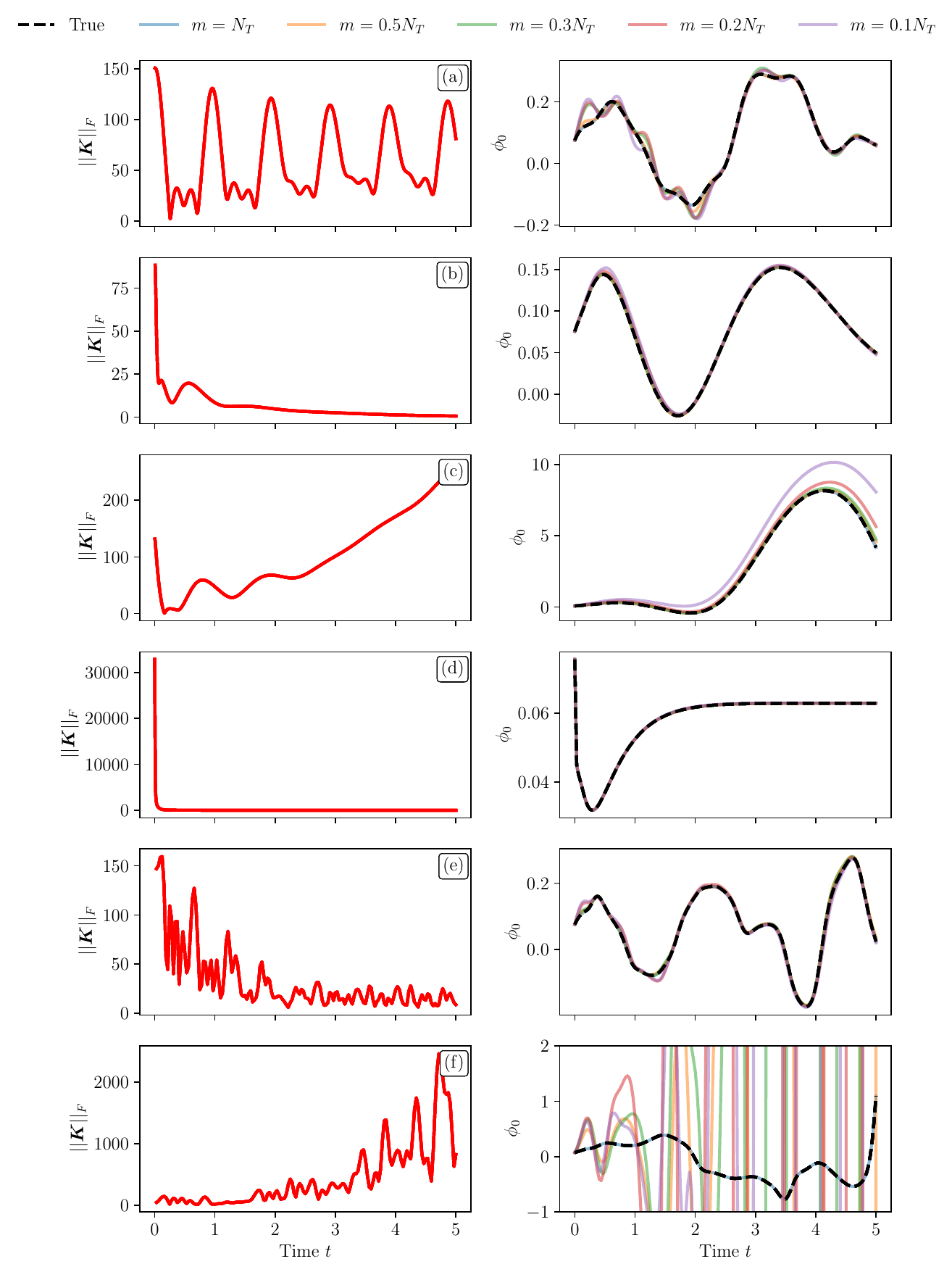}
    \caption{Predictions under a finite-memory approximation. Left: time evolution of the Frobenius norm $||~K||_F$ of the reconstructed full-memory kernel. Right: predicted trajectories of the first component of $\phi$ for a single test initial condition in the first six reaction-diffusion-advection test cases, with $\Delta t = 0.03125$. The operators were reconstructed using full observation data. Finite memory is enforced by truncating the kernel support, setting $K_n=0$ for all lags $n>m$, so that only contributions up to $t_m=m\Delta t$ are retained; $m=N_T$ recovers the full-memory model. In (a) and (c), trajectory accuracy degrades as the truncation becomes more severe (smaller $m$), whereas in (f) the trajectories are not predicted accurately for any truncation level. This behavior is attributable to non-decaying memory kernels, indicating a strong influence of the memory term, particularly in case (f).}
    \label{fig:pde_PH}
\end{figure}

\begin{table}
    \caption{Relative errors in the predicted trajectories $~\Phi^{\text{pred}}$ for the first six reaction-diffusion-advection test cases under a finite-memory approximation of the kernel $~K$, with $~K_n = 0$ for all $n > m$ ($m = N_T$ corresponds to the full-memory case).} \label{tab:eq1_PH}
\centering
\begin{tabular}{c| c c c c c c}
\toprule
 Case & $m=0.1N_T$ & $m=0.2N_T$ & $m=0.3N_T$ & $m=0.5N_T$ & $m=N_T$ \\
\midrule
(a) & $4.64\times10^{-2}$ & $6.13\times10^{-2}$ & $6.60\times10^{-2}$ & $2.01\times10^{-2}$ & $3.86\times10^{-9}$ \\
(b) & $1.17\times10^{-2}$ & $1.09\times10^{-2}$ & $7.05\times10^{-3}$ & $2.44\times10^{-3}$ & $1.70\times10^{-12}$ \\
(c) & $2.42\times10^{-1}$ & $1.60\times10^{-1}$ & $4.77\times10^{-2}$ & $1.62\times10^{-2}$ & $1.86\times10^{-5}$ \\
(d) & $9.77\times10^{-4}$ & $2.76\times10^{-4}$ & $7.96\times10^{-5}$ & $6.04\times10^{-6}$ & $4.64\times10^{-12}$ \\
(e) & $2.27\times10^{-2}$ & $3.39\times10^{-2}$ & $1.68\times10^{-2}$ & $1.78\times10^{-2}$ & $1.22\times10^{-11}$ \\
(f) & $ \ge 1$ & $ \ge 1$ & $ \ge 1$ & $ \ge 1$ & $4.57\times10^{-3}$ \\
\end{tabular}
\end{table}

\begin{table}
\caption{Relative errors in the predicted trajectories $~\Phi^{\text{pred}}$ for the two reaction-diffusion-advection test cases with time-varying $~A(t)$. Operators $~R$, $~K$, $~B$ were reconstructed with partial data by solving the regularized problem \eqref{eq:localpartialreg} for different values of the regularization parameters $\lambda_{~R}$ and $\lambda_{~K}$.} \label{tab:eq1_reg}
\centering
\begingroup
\begin{tabular}{cc*{5}{c}}
\toprule
Case  & {$\lambda_{~R} \backslash \lambda_{~K}$}
   & {$0$} & {$10^{-8}$} & {$10^{-4}$} & {$10^{-2}$} & {$10^{0}$} \\
\midrule
(e) & 0        & $ \ge 1$           & $ \ge 1$           & $ \ge 1$           & $ \ge 1$           & $ \ge 1$ \\
    & {$10^{-8}$} & $5.59\times10^{-7}$ & $ \ge 1$           & $ \ge 1$           & $ \ge 1$           & $ \ge 1$ \\
    & {$10^{-4}$} & $5.28\times10^{-7}$ & $1.15\times10^{-11}$ & $ \ge 1$         & $ \ge 1$           & $ \ge 1$ \\
    & {$10^{-2}$} & $4.61\times10^{-7}$ & $1.15\times10^{-11}$ & $2.35\times10^{-4}$ & $ \ge 1$       & $ \ge 1$ \\
    & {$10^{0}$}  & $4.07\times10^{-7}$ & $1.15\times10^{-11}$ & $2.33\times10^{-4}$ & $2.35\times10^{-2}$ & $3.07\times10^{-2}$ \\
\midrule
(f) & 0        & $ \ge 1$           & $ \ge 1$           & $ \ge 1$           & $ \ge 1$           & $ \ge 1$ \\
    & {$10^{-8}$} & $8.16\times10^{-7}$ & $ \ge 1$           & $ \ge 1$           & $ \ge 1$           & $ \ge 1$ \\
    & {$10^{-4}$} & $1.09\times10^{-6}$ & $1.48\times10^{-11}$ & $ \ge 1$         & $ \ge 1$           & $ \ge 1$ \\
    & {$10^{-2}$} & $7.08\times10^{-7}$ & $1.48\times10^{-11}$ & $ \ge 1$         & $ \ge 1$           & $ \ge 1$ \\
    & {$10^{0}$}  & $1.50\times10^{-6}$ & $1.48\times10^{-11}$ & $6.87\times10^{-4}$ & $3.97\times10^{-2}$ & $1.96\times10^{-1}$ \\
\end{tabular}
\endgroup
\end{table}

\section{Conclusions}\label{s;conclusions}

Using the Mori-Zwanzig formalism, we derived an exact model for the projected reduced dynamics of generic high-dimensional driven linear systems and posed inverse problems to learn the associated model operators from full and partial observation data. By introducing a greedy, sequential-in-time solution strategy, we identified conditions under which these operators are identifiable and the inverse problems are well posed. Numerical experiments on a one-dimensional reaction–diffusion–advection equation and a damped wave equation demonstrate that the resulting data-driven, history-enriched models faithfully reproduce the dynamics of the resolved degrees of freedom.

However, the proposed framework has several limitations. Our analysis is restricted to linear dynamical systems. Extending the formulation and identifiability theory to parametric and nonlinear systems would significantly broaden the range of applications, and we hope the present work will help guide the construction and analysis of inverse problems in those settings. A systematic treatment of noisy and sparse observation data—together with more robust regularization strategies—remains to be developed. These issues lie beyond the scope of this paper but represent natural next steps toward practical, history-enriched reduced models for more complex dynamical systems.

\section*{Acknowledgments}
This work was supported by the National Science Foundation under Cooperative Agreement 2421782 and the Simons Foundation award MPS-AI-00010515(NSF-Simons AI Institute for Cosmic Origins - CosmicAI, \url{https://www.cosmicai.org/}). A. Vijaywargiya also acknowledges support from the Peter O’Donnell Postdoctoral Fellowship awarded by the Oden Institute of Computational Engineering and Sciences.

\appendix 
\begin{appendices}

\section[Proof of Theorem~\ref{thm:type1}]{Proof of \Cref{thm:type1}}\label{app:proof}

Before establishing the well-posedness of the Full Data inverse problem, we need the following lemma.

\begin{lemma}\label{lemma:bound}
    Let $~X \in \mathbb R^{n \times p}$ solve the least-squares problem 
    \begin{align*}
        \underset{~X}{\operatorname{argmin}} \;\|~Z - ~F~X\|_F^2, 
    \end{align*}
    where $~Z \in \mathbb {R}^{m \times p}$ and  $~F \in \mathbb R^{m \times n}$ has full column rank. 
    If $\hat{~X}$ solves the perturbed problem 
    \begin{align*}
        \underset{\hat{~X}}{\operatorname{argmin}} \;\|\hat{~Z} - \hat{~F}\hat{~X}\|_F^2,
    \end{align*}
    where $\hat{~Z} := ~Z + \Delta ~Z \in \mathbb {R}^{m \times p}$ and 
    $\hat{~F} := ~F + \Delta ~F \in \mathbb R^{m \times n}$ with 
    $\|\Delta ~F\|_2 < \sigma_{\min}(~F)$, then 
    \begin{align}\label{eq:rel_err_bound}
        \frac{\|\hat{~X} - ~X\|_F}{\|~X\|_F}
        \leq 
        \frac{\kappa_2(~F)}{1-\kappa_2(~F)\cdot\frac{\|\Delta ~F\|_2}{\|~F\|_2} }
        \left(
            \frac{\|\Delta ~Z\|_F}{\|~F\|_2\|~X\|_F } 
            + \frac{\|\Delta ~F\|_2}{\|~F\|_2} 
            + \frac{\|~O\|_F}{\|~F\|_2 \|~X\|_F}
        \right),
    \end{align}
    where $~O := ~Z - ~F~X$.
\end{lemma}

\begin{proof}
By \cite[Corollary 7.3.5]{HornJohnson2013},
\[
    \big|\sigma_{\min} (~F) - \sigma_{\min} (\hat{~F})\big|
    \le \|\Delta ~F\|_2
    \;\implies\;
    \sigma_{\min}(\hat{~F}) \ge \sigma_{\min}(~F) - \|\Delta ~F\|_2 > 0.
\]
Thus, $~F$ and $\hat{~F}$ have full column rank, and the pseudoinverses 
$~F^\dagger := (~F^\top ~F)^{-1} ~F^\top$ and 
$\hat{~F}^\dagger := (\hat{~F}^\top \hat{~F})^{-1} \hat{~F}^\top$ are well defined.
Since $~Z = ~F~X + ~O$ and $\hat{~Z} = ~Z + \Delta ~Z$, we have
\[
    \hat{~X} - ~X
    = \hat{~F}^\dagger(\Delta ~Z - \Delta ~F\,~X + ~O),
\]
and therefore
\[
    \|\hat{~X} - ~X\|_F
    \le \|\hat{~F}^\dagger\|_2
    \bigl(\|\Delta ~Z\|_F + \|\Delta ~F\|_2\|~X\|_F + \|~O\|_F\bigr).
\]
Dividing by $\|~X\|_F$ and using 
$\|\hat{~F}^\dagger\|_2 = 1/\sigma_{\min}(\hat{~F})
 \le 1/(\sigma_{\min}(~F)-\|\Delta ~F\|_2)$
gives
\[
    \frac{\|\hat{~X} - ~X\|_F}{\|~X\|_F}
    \le 
    \frac{\|~F\|_2}{\sigma_{\min}(~F) - \|\Delta ~F\|_2}
    \left(
        \frac{\|\Delta ~Z\|_F}{\|~F\|_2\|~X\|_F}
        + \frac{\|\Delta ~F\|_2}{\|~F\|_2}
        + \frac{\|~O\|_F}{\|~F\|_2\|~X\|_F}
    \right).
\]
Finally, $\|~F\|_2 = \kappa_2(~F)\sigma_{\min}(~F)$ yields \eqref{eq:rel_err_bound}.
\end{proof}

\subsection[Proof of Theorem~\ref{thm:type1}]{Proof of \Cref{thm:type1}}

\begin{proof}
$\\$
    \begin{itemize}
        \item \textbf{Uniqueness}

        Let $~\Psi(t)$ denote the matrix of the flow map associated with \eqref{eq:cont_dyn}, so that
    $~u(t) = ~\Psi(t)\,~u(0)$. Since $[~\phi(t); \tilde{~\phi}(t)]$ is obtained from
    $~u(t)$ by an invertible linear change of coordinates, there exists an invertible matrix
    $\bar{~\Psi}(t)$ in the $[~\phi;\tilde{~\phi}]$-coordinates such that, for each $n$,
    \[
        ~F_{n+1}^R = ~F_0\,\bar{~\Psi}(t_{n+1})^\top.
    \]
    Right multiplication by the invertible matrix $\bar{~\Psi}(t_{n+1})^\top$ preserves column rank, by assumption (i) each $~F_{n+1}^R$ has full column rank. Vertical stacking does not change the column
    space, and therefore $~F^R$ also has full column rank. It is easily verified that under (ii)---(iii) and the above analysis, $~F^{KB}$ has full column rank as well. Therefore, \eqref{eq:Rn}, \eqref{eq:R}, and \eqref{eq:KB} each admit a unique minimizer.

    \item \textbf{Stability}
        
        Under small perturbations, the bound for the relative errors in the solutions to \eqref{eq:Rn}, \eqref{eq:R}, and \eqref{eq:KB} follow directly from \Cref{lemma:bound} with $~O = ~0$ for \eqref{eq:Rn} (for each $n$) and \eqref{eq:R} thanks to model consistency.
    
    \item \textbf{Condition number bound }

    The derivative of $||~x(t)||_2^2$ can be computed as
    \begin{align*}
        \frac{d}{dt}||~x(t)||_2^2 = \frac{d}{dt} ~x(t)^\star ~x(t) = ~x(t)^\star~A(t) ~x(t) + ~x(t)^\star ~A(t)^\star ~x(t) = 2~x(t)^\star ~H(t) ~x(t).
    \end{align*}
    Applying Rayleigh quotient bounds, we obtain
    \begin{align*}
        \lambda_{\min}(~H(t)) ||~x(t)||_2^2 \leq  ~x(t)^\star ~H(t) ~x(t) \leq \lambda_{\max}(~H(t)) ||~x(t)||_2^2.
    \end{align*}
    Hence,
    \begin{align*}
        2\lambda_{\min}(~H(t)) ||~x(t)||_2^2 \leq  \frac{d}{dt} ~x(t)^\star ~x(t) \leq 2\lambda_{\max}(~H(t)) ||~x(t)||_2^2.
    \end{align*}
    Applying Grönwall's inequality, we get, for all $t \ge 0$,
\begin{align}\label{eq:gronwall}
    e^{\int_0^t \lambda_{\min}(~H(s))ds}||~x_0||_2  \leq ||~x(t)||_2 \leq e^{\int_0^t \lambda_{\max}(~H(s))ds}||~x_0||_2 
\end{align}
Substituting this expression in the above inequality and taking the maximum and minimum over all unit $~x(0)$ gives the following bounds for the extremal singular values of $~\Psi(t)$:
\begin{align*}
    \sigma_{\max}(~\Psi(t)) \leq e^{\int_0^t \lambda_{\max}(~H(s))ds}, \qquad \sigma_{\min}(~\Psi(t)) \geq e^{\int_0^t \lambda_{\min}(~H(s))ds}.
\end{align*}
Hence, the 2-norm condition number of $~\Psi(t)$ satisfies the bound
\begin{align*}
    \kappa_2(~\Psi(t)) \leq e^{\int_0^t \lambda_{\max}(~H(s)) -\lambda_{\min}(~H(s))ds}
\end{align*}
Using $~F^R_{n+1} = ~F_0 ~\Psi({t_{n+1}})^\top$, the above bound, and the submultiplicativity of the 2-norm condition number, the bound for $\kappa_2(~F^R_{n+1})$ follows. Now for any unit $~v \in \mathbb R^N$, 
        \begin{align*}
            ||~F^R_{n+1} ~v||_2 &= || ~F_0 ~\Psi(t_{n+1})^\top ~v||_2 
            \geq \sigma_{\min}(~F_0) \sigma_{\min}(~\Psi(t_{n+1})) || ~v||_2 
            \geq \sigma_{\min}(~F_0) e^{\int_0^{t_{n+1}} \lambda_{\min}(~H(s)) ds}
        \end{align*}
        \begin{align*}
            ||~F^R_{n+1} ~v||_2 &= || ~F_0 ~\Psi(t_{n+1})^\top ~v||_2 
            \leq \sigma_{\max}(~F_0) \sigma_{\max}(~\Psi(t_{n+1})) || ~v||_2 
            \leq \sigma_{\max}(~F_0) e^{\int_0^{t_{n+1}} \lambda_{\max}(~H(s))ds}
        \end{align*}
        By the variational characterization of singular values,
        \begin{align*}
            \sigma_{\min}(~F^R)^2 = \min_{||~v|| = 1} \sum_{n=0}^{N_T-1} ||~F_{n+1} ~v||_2^2 \geq \sum_{n=0}^{N_T-1} \sigma_{\min}(~F_0)^2 e^{2\int_0^{t_{n+1}} \lambda_{\min}(~H(s)) ds},
        \end{align*}
        \begin{align*}
            \sigma_{\max}(~F^R)^2 = \max_{||~v|| = 1} \sum_{n=0}^{N_T-1} ||~F_{n+1}~v||_2^2 
            \leq \sum_{n=0}^{N_T-1} \sigma_{\max}(~F_0)^2 e^{2\int_0^{t_{n+1}} \lambda_{\max}(~H(s))ds},
        \end{align*}
        from which the bound for $\kappa_2(~F^R)$ can be obtained. The bound for $\kappa_2(~F^{KB})$ follows using a similar argument.
    \end{itemize}
\end{proof}

\section{Alternative discretizations}\label{app:altdisc}

\subsection{Implicit midpoint rule}\label{sec:impmid}
We provide a second-order accurate discretization of \eqref{eq:vide}, and inverse problems in \cref{type:inv1} and \cref{type:inv2}. Applying the implicit midpoint rule to \eqref{eq:vide}, one obtains
\begin{align}\label{eq:impl}
    \frac{~\phi_{n+1} - ~\phi_{n}}{\Delta t} = ~R_{n+\tfrac12} ~\phi_{n+\tfrac12} + ~B_{n+\tfrac12}\tilde{~\phi}_0 + ~g_{n+\tfrac12} 
    &+ \int_{0}^{t_{n+1/2}} ~K(t_{n+\tfrac12} - s) ~\phi(s) ds \\
    &+ \int_{0}^{t_{n+1/2}} ~B(t_{n+\tfrac12} - s) \tilde{~g}(s) ds
\end{align}
The two integrals are approximated as
\begin{align}
\int_{0}^{t_{n+1/2}} ~K(t_{n+\tfrac12} - s) ~\phi(s) ds 
    &\approx 
    \Delta t \sum_{k=0}^{n-1} ~K(t_{n+\tfrac12} - t_{k+\tfrac12}) ~\phi_{k+\tfrac12} + \frac{\Delta t}{2} ~K(t_{n+\tfrac12} - t_{n+\tfrac12}) ~\phi_{n+\tfrac12}\nonumber \\
    &= \Delta t \sum_{k=0}^{n-1} ~K_{n-k} ~\phi_{k+\tfrac12} + \frac{\Delta t}{2} ~K_0 ~\phi_{n+\tfrac12}, \label{eq:impintK}\\
\int_{0}^{t_{n+1/2}} ~B(t_{n+\tfrac12} - s) \tilde{~g}(s) ds 
&\approx 
\frac{\Delta t}{2} \sum_{k=0}^{n} ~B(t_{n+\tfrac12} - t_k) \tilde{~g}_k + \frac{\Delta t}{2} \sum_{k=1}^{n} ~B(t_{n+\tfrac12} - t_k) \tilde{~g}_k \nonumber \\
&= \frac{\Delta t}{2}~B_{n+\tfrac12} \tilde{~g}_0 + {\Delta t} \sum_{k=1}^{n} ~B_{n-k+\tfrac12} \tilde{~g}_k. \label{eq:impintB}
\end{align}
Expression \eqref{eq:impintK} uses a composite midpoint rule, except in the last (half-) cell, where a left-endpoint evaluation is applied. Expression \eqref{eq:impintB} uses a composite endpoint rule, employing two endpoint evaluations in each cell except the last. Substituting these in \eqref{eq:impl} and rearranging, the update for $~\phi_{n+1}$ is written as follows:

\begin{align}\label{eq:implicit_midpoint}
    ~\phi_{n+1} = \left[~I-\frac{\Delta t}{2}~R_{n+\tfrac12} - \frac{\Delta t^2 }{4} ~K_0 \right]^{-1}
    &\left[
    \left( ~I + \frac{\Delta t}{2} ~R_{n+\tfrac12} +\frac{\Delta t^2}{4} ~K_0 \right) \right. ~\phi_n 
    + \Delta t ~g_{n+\tfrac12}
    + \Delta t ~B_{n+\tfrac12}\tilde{~\phi}_0 \nonumber \\ 
    &\left. \hspace{15mm}
        + \Delta t^2\sum_{k=0}^{n-1}~K_{n-k}\frac{~\phi^k + ~\phi^{k+1}}{2} 
        + \Delta t^2\sum_{k=0}^{n-1}~B_{n-k+\tfrac12}\tilde{~g}_k
    \right],
\end{align}
where $~I$ denotes the identity matrix. The inverse problems that reconstruct the operators $~R, ~K$, and $~B$ based on \eqref{eq:impl} are stated in the two definitions below.

\begin{defn}{(Discrete Full Data inverse problem; implicit midpoint rule)}
Given full trajectory data $\{\big(~\Phi_n, \tilde{~\Phi}_n\big)\}_{n=0}^{N_T}$ and forcing data $\{(~G_n, \tilde{~G_n})\}_{n=0}^{N_T}$, the set of Markovian operators $\{~R_{n+\frac{1}{2}}\}_{n=0}^{N_T-1}$ and $\{\widetilde{~R}_{n+\frac{1}{2}}\}_{n=0}^{N_T-1}$ can be reconstructed by solving
\begin{subequations}
\begin{align}
 \underset{~R_{n+\frac{1}{2}}, \tilde{~R}_{n+\frac{1}{2}}}{\operatorname{argmin}} \; 
 \frac{\Delta t}{2N_s} \sum_{n=0}^{N_T-1} 
\left\| 
    \dot{~\Phi}_{n+\tfrac12} 
    - ~G_{n+\tfrac12}
    - ~\Phi_{n+\frac{1}{2}}~R_{n+\frac{1}{2}}^\top 
    - \tilde{~\Phi}_{n+\frac{1}{2}}\tilde{~R}^\top_{n+\frac{1}{2}} 
\right\|_F^2,
\end{align}    
where $~R^\top$ denotes the transpose of $~R$, $~\Phi_{n+\frac{1}{2}} = \frac{1}{2}(~\Phi_{n}+~\Phi_{n+1})$, and $~G_{n+\frac{1}{2}} = \frac{1}{2}(~G_{n} + ~G_{n+1})$.  
Matrices $\{~K_{n}\}_{n=0}^{N_T-1}$ and $\{~B_{n+\frac{1}{2}}\}_{n=0}^{N_T-1}$ can subsequently be reconstructed by solving 
\begin{align}
     \underset{~K_{n}, ~B_{n+\frac{1}{2}}}{\operatorname{argmin}}\; 
     \frac{\Delta t}{2N_s} \sum_{n=0}^{N_T-1}  
\left\| 
    ~L_{n+\frac{1}{2}}  
    - (\tilde{~\Phi}_{0} + \frac{\Delta t}{2}\tilde{~G}_0 ){~B}^\top_{n+\frac{1}{2}}
    - \sum_{k=0}^{n} \Delta t_{nk} ~\Phi_{k+\tfrac{1}{2}} ~K_{n-k}^\top 
    - \Delta t \sum_{k=1}^{n} \tilde{~G}_{k} ~B_{n-k+\tfrac12}^\top 
\right\|_F^2 \,,
\end{align}
where $~L_{n+\frac{1}{2}} = \dot{~\Phi}_{n+\tfrac12} - ~G_{n+\tfrac12} - ~\Phi_{n+\tfrac{1}{2}}~R_{n+\frac{1}{2}}^\top$, $\Delta t_{nk} = \Delta t$ for $k < n$, and $\Delta t_{nn} = \tfrac{\Delta t}{2}$.
\end{subequations}
\end{defn}

\begin{defn}{(Discrete Partial Data inverse problem; implicit midpoint rule)}
Given partial trajectory data $\{\big(~\Phi_n, \tilde{~\Phi}_0\big)\}_{j=1}^{N_s}$ and forcing data $\{(~G_n, \tilde{~G_n})\}_{n=0}^{N_T}$, matrices $\{~R_{n+\frac{1}{2}}, ~K_{n}, ~B_{n+\frac{1}{2}}\}_{n=0}^{N_T-1}$ can be reconstructed by solving
\begin{align}
    \underset{~R_{n+\frac{1}{2}}, ~K_{n}, ~B_{n+\frac{1}{2}}} {\operatorname{argmin}}\; \frac{\Delta t}{2N_s} \sum_{n=0}^{N_T-1} 
\left\| 
\dot{~\Phi}_{n+\tfrac12} - ~G_{n+\tfrac12} - ~\Phi_{n+\tfrac{1}{2}}~R_{n+\frac{1}{2}}^\top
- (\tilde{~\Phi}_{0} + \frac{\Delta t}{2}\tilde{~G}_0 ){~B}^\top_{n+\frac{1}{2}}
\right. & \nonumber \\ & \hspace{-60mm}
\left. 
- \sum_{k=0}^{n} \Delta t_{nk} ~\Phi_{k+\tfrac{1}{2}} ~K_{n-k}^\top 
- \Delta t \sum_{k=1}^{n} \tilde{~G}_{k} ~B_{n-k+\tfrac12}^\top 
\right\|_F^2,
\end{align}
where $\Delta t_{nk} = \Delta t$ for $k < n$ and $\Delta t_{nn} = \tfrac{\Delta t}{2}$.
\end{defn}

\subsection{Forward Euler}

Due to an implicit discretization, \eqref{eq:time-march} involves a matrix inversion at each time step. An alternative discretization that is also first order accurate but does not require a matrix inversion is provided by applying the forward Euler rule to \eqref{eq:vide} as follows:
\begin{align}\label{eq:backeul}
    \frac{~\phi_{n+1} - ~\phi_{n}}{\Delta t} = ~R_n ~\phi_n + ~B_n\tilde{~\phi}_0 + ~g_{n} 
    + \int_{0}^{t_{n}} ~K(t_{n} - s) ~\phi(s) ds 
    + \int_{0}^{t_{n}} ~B(t_{n} - s) \tilde{~g}(s) ds
\end{align}
The two integrals are approximated as
\begin{align}
    &\int_{0}^{t_{n}} ~K(t_{n} - s) ~\phi(s) ds 
    \approx
    \Delta t \sum_{k=0}^{n-1} ~K(t_n - t_k) ~\phi_k = \Delta t \sum_{k=0}^{n-1} ~K_{n - k} ~\phi_k, \\
    &\int_{0}^{t_{n}} ~B(t_{n} - s) \tilde{~g}(s) ds 
    \approx
    \Delta t \sum_{k=0}^{n-1} ~B(t_n - t_k) \tilde{~g}_k = \Delta t \sum_{k=0}^{n-1} ~B_{n - k} \tilde{~g}_k,
\end{align}
each using a composite left-end point rule. Substituting in \eqref{eq:backeul} and rearranging yields the update
\begin{align}
    ~\phi_{n+1} = ~\phi_n + \Delta t (~R_n ~\phi_n + ~B_n\tilde{~\phi}_0 + ~g_{n}) + \Delta t^2 \sum_{k=0}^{n-1} ~K_{n - k} ~\phi_k + \Delta t^2 \sum_{k=0}^{n-1} ~B_{n - k} \tilde{~g}_k.
\end{align}
The inverse problems for $~R, ~K$, and $~B$ based on \eqref{eq:backeul} are given in the following definitions.

\begin{defn}{(Discrete Full Data inverse problem; forward Euler)}
Given full trajectory data $\{\big(~\Phi_n, \tilde{~\Phi}_n\big)\}_{n=0}^{N_T}$ and forcing data $\{(~G_n, \tilde{~G_n})\}_{n=0}^{N_T}$, the set of Markovian operators $\{~R_{n}\}_{n=0}^{N_T-1}$ can be reconstructed by solving
\begin{subequations}
\begin{align}
 \underset{~R_{n}, \tilde{~R}_{n}}{\operatorname{argmin}} \; 
 \frac{\Delta t}{2N_s} \sum_{n=0}^{N_T-1} 
\left\| 
    \dot{~\Phi}_{n} 
    - ~G_{n}
    - ~\Phi_{n}~R_{n}^\top 
    - \tilde{~\Phi}_{n}\tilde{~R}^\top_{n} 
\right\|_F^2,
\end{align}    
where $~R^\top$ denotes the transpose of $~R$. Matrices $\{~K_{n}\}_{n=0}^{N_T-1}$ and $\{~B_{n}\}_{n=0}^{N_T-1}$ are subsequently reconstructed by solving 
\begin{align}
     \underset{~K_{n}, ~B_{n}}{\operatorname{argmin}}\; 
     \frac{\Delta t}{2N_s} \sum_{n=0}^{N_T-1}  
\left\| 
    ~L_{n}  
    - \tilde{~\Phi}_{0} {~B}^\top_{n}
    - \Delta t \sum_{k=0}^{n-1} ~\Phi_{k} ~K_{n-k}^\top 
    - \Delta t \sum_{k=0}^{n-1} \tilde{~G}_{k} ~B_{n-k}^\top 
\right\|_F^2 \,,
\end{align}
where $~L_{n} = \dot{~\Phi}_{n} - ~G_{n} - ~\Phi_{n}~R_{n}^\top$.
\end{subequations}
\end{defn}

\begin{defn}{(Discrete Partial Data inverse problem; forward Euler)}
Given partial trajectory data $\{\big(~\Phi_n, \tilde{~\Phi}_0\big)\}_{j=1}^{N_s}$ and forcing data $\{(~G_n, \tilde{~G_n})\}_{n=0}^{N_T}$, matrices $\{~R_{n}, ~K_{n}, ~B_{n}\}_{n=0}^{N_T-1}$ can be reconstructed by solving
\begin{align}
    \underset{~R_{n}, ~K_{n}, ~B_{n}} {\operatorname{argmin}}\; \frac{\Delta t}{2N_s} \sum_{n=0}^{N_T-1} 
\left\| 
\dot{~\Phi}_{n} - ~G_{n} - ~\Phi_{n}~R_{n}^\top
- \tilde{~\Phi}_{0} {~B}^\top_{n}
-  \Delta t\sum_{k=0}^{n-1} ~\Phi_{k} ~K_{n-k}^\top 
- \Delta t\sum_{k=0}^{n-1}  \tilde{~G}_{k} ~B_{n-k}^\top 
\right\|_F^2.
\end{align}
\end{defn}

\section{Pseudospectral discretization}\label{app:pseudo}

We consider spatially $2\pi$-periodic problems and discretize the spatial variable by a Fourier pseudospectral method. Let $N$ be an even integer and introduce the equispaced grid points
$$
x_j = \frac{2\pi j}{N}, \qquad j = 0,\dots,N-1.
$$
Given a periodic function $u(t,x)$ we denote its nodal values by
$$
u_j(t) = u(t,x_j), \qquad j=0,\dots,N-1,
$$
and collect them in the vector
$$
~{u}(t) = (u_0(t),\dots,u_{N-1}(t))^{\top} \in \mathbb{C}^N.
$$
The spatial discretization is formulated in terms of the unitary discrete Fourier transform. Let
$$
K_N = \Bigl\{-\tfrac{N}{2}+1,\dots,-1,0,1,\dots,\tfrac{N}{2}\Bigr\}
$$
be a set of $N$ distinct integer wavenumbers. For a given vector $\mathbf{u}(t)$ we define its discrete Fourier coefficients by
$$
\widehat{u}_k(t) = \frac{1}{\sqrt{N}}\sum_{j=0}^{N-1} u_j(t)\,e^{-ikx_j},
\qquad k\in K_N,
$$
with the inverse relation
$$
u_j(t) = \frac{1}{\sqrt{N}}\sum_{k\in K_N} \widehat{u}_k(t)\,e^{ikx_j},
\qquad j=0,\dots,N-1.
$$
Let $~F\in\mathbb{C}^{N\times N}$ denote the unitary discrete Fourier matrix with entries
$$
F_{k,j} = \frac{1}{\sqrt{N}} e^{-ikx_j},
\qquad k\in K_N,\; j=0,\dots,N-1.
$$
Then
$$
\widehat{~{u}}(t) = ~F ~{u}(t),
\qquad
~{u}(t) = ~F^{*} \widehat{~{u}}(t),
$$
where $~F^{*}$ denotes the conjugate transpose of $~F$. Differentiation in Fourier space is diagonal. We introduce the diagonal matrices
$$
~\Lambda = \mathrm{diag}(ik)_{k\in K_N}, 
\qquad
~\Lambda^{(2)} = \mathrm{diag}(-k^2)_{k\in K_N},
$$
so that
$$
\widehat{\partial_x u}(t) = \Lambda\,\widehat{~{w}}(t),
\qquad
\widehat{\partial_{xx} u}(t) = \Lambda^{(2)}\,\widehat{~{w}}(t).
$$
Transforming back to physical space yields
$$
(\partial_x u)(t) = ~F^{*} ~\Lambda ~F\,~{u}(t),
\qquad
(\partial_{xx} u)(t) = ~F^{*} ~\Lambda^{(2)} ~F\,~{u}(t).
$$
This defines the spectral differentiation matrices
\begin{align}\label{eq:diffmatrices}
 ~{D} := ~F^{*}~\Lambda ~F,
\qquad
~{D}^{(2)} := ~F^{*}~\Lambda^{(2)} ~F,
\end{align}
which represent the pseudospectral approximations of the first and second spatial derivatives at the collocation points.

The discretizations \eqref{eq:ode1} and \eqref{eq:ode2} follow by enforcing the PDEs \eqref{eq:pde1} and \eqref{eq:damped_wave} at the collocation points $x_j$, with each spatial derivative replaced by its pseudospectral approximation via the differentiation matrices defined in \eqref{eq:diffmatrices}.

\section{Numerical example: Damped wave equation}\label{app:wave}

Consider the initial-boundary value problem
\begin{equation}
\label{eq:damped_wave}
\left\{
\begin{aligned}
\partial_{tt} u(t,x) 
&\;+\; \big(\gamma_1(t)+\gamma_2(t)\big)\,\partial_t u(t,x)
\\
&\;+\; \big(\dot\gamma_1(t)+\gamma_1(t)\gamma_2(t)\big)\, u(t,x)
   \;=\; c^2\,\partial_{xx} u(t,x),
   && (t,x)\in[0,T]\times[0,2\pi],\\[0.3em]
u(0,x) &= u_0(x), \quad \partial_t u(0,x)=\dot u_0(x),
   && x\in[0,2\pi],\\[0.3em]
u(t,0) &= u(t,2\pi), && t \in [0,T].
\end{aligned}
\right.
\end{equation}
Equation~\eqref{eq:damped_wave} represents a damped wave equation with time-dependent damping coefficients $\gamma_1(t)$ and $\gamma_2(t) > 0$ and constant wave speed $c$. No external forcing is present in this system. 
A Fourier pseudospectral discretization on the same periodic grid used previously yields the semi-discrete system
\begin{align}\label{eq:ode2}
  \ddot{~u}(t)
  + \big(\gamma_1(t) + \gamma_2(t)\big)\,\dot{~u}(t)
  + \big(\dot{\gamma}_1(t) + \gamma_1(t)\gamma_2(t)\big)\,~u(t)
  = c^2 ~{D}^2 ~u(t),
\end{align}
where $~u(t)\in\mathbb R^N$ is the state vector and $~D$ is the spectral differentiation matrix. Introducing the auxiliary variables $~\phi(t), \tilde{~\phi}(t) \in \mathbb R^N$
\begin{align*}
  ~\phi(t) := ~u(t), \qquad
  \tilde{~\phi}(t) := \dot{~\phi}(t) + \gamma_1(t)\,~\phi(t),
\end{align*}
the semi-discrete system \eqref{eq:ode2} can be written in first-order form as
\begin{align}\label{eq:damped_wave_first_order}
  \begin{bmatrix} \dot{~\phi} \\ \dot {\tilde{~\phi}} \end{bmatrix}
  =
  \begin{bmatrix} -\gamma_1(t) ~I & ~I \\ c^2~D^2  & -\gamma_2(t) ~I \end{bmatrix}
  \begin{bmatrix} ~\phi \\ \tilde{~\phi} \end{bmatrix},
\end{align}
where $~I \in \mathbb{R}^{N \times N}$ is the identity matrix. Applying the MZ formalism and eliminating $\tilde{~\phi} $ gives the following VIDE for $~\phi$:
\begin{align}\label{eq:waveMZ}
  \dot{~\phi} (t) = \underbrace{-\gamma_1(t) ~I}_{~R(t)} ~\phi(t) 
  + \underbrace{e^{-\Gamma(t)}~I}_{~B(t)}\tilde{~\phi}_0 
  + \int_0^t \underbrace{e^{-(\Gamma (t) - \Gamma(s))}c^2 ~D^2}_{~K(t,s)} ~\phi(s) ds
\end{align}
where $\Gamma(t) := \int_0^t \gamma_2(s)\,ds$. The memory kernel $~K(t,s)$ is
stationary, i.e., depends only on the time lag $t-s$, if and only if
$\gamma_2(t)$ is constant in time. The goal of this experiment is to obtain
accurate trajectories of $~\phi$ using the MZ model \eqref{eq:waveMZ}, with
operators reconstructed for different choices of the damping coefficients. We
study three cases:
\begin{enumerate}[label=(\alph*)]
    \item $\gamma_1 = 0.25$ and $\gamma_2 = 0.25$;
    \item $\gamma_1 = 0.5\sin^2(t)$ and $\gamma_2 = 1$;
    \item $\gamma_1 = 0.5\cos(t)$ and $\gamma_2 = 1 + \sin(t)$.
\end{enumerate}
Cases (b) and (c) have time-varying parameters. In case (b), $\gamma_2$ is
constant so the memory kernel remains stationary, whereas in case (c) both
$\gamma_1$ and $\gamma_2$ vary in time and $~K(t,s)$ is non-stationary. We set $T = 2\pi$ and $c = 1$. Choosing $N = 60$ produces a first-order system
\eqref{eq:damped_wave_first_order} with $2N = 120$ degrees of freedom. A total
of $N_s = 120$ independent initial conditions are generated for training and
30 for testing, and the data matrices $~\Phi_n$ and $\tilde{~\Phi}_n$ are
constructed exactly as in the previous experiment.

As before, we first reconstruct the operators using full observation data,
compute trajectories of the resolved variable $~\phi$ on a hierarchy of time
grids, and evaluate the associated relative errors. The results, summarized in
\Cref{tab:eq_2_conv}, show second-order convergence of $~R$ and $~K$. The
prediction error in $~\Phi^\text{pred}$ remains consistent with the solver
tolerance $\tau_{\text{sol}}$ even for case (c), which has a non-stationary
memory kernel.


Next, we assess prediction under a finite-memory approximation of the memory kernel. The time step is fixed at $\Delta t = 0.0625$, and a total of 2000 \texttt{lsqr} iterations are performed to reconstruct $~K$ through \eqref{eq:KB_PH}. The finite-memory parameter $m$ is varied as in \Cref{sec:pde1}, and the resulting relative errors in $~\Phi^\text{pred}$ are reported in \Cref{tab:eq2_PH}. Because the decay of the memory kernel is governed solely by $\gamma_2$, the damping is strong enough in all cases to keep the error below $5\%$ when $m = 0.5N_T$. For case (c), the error remains below $5\%$ even for $m = 0.1N_T$, whereas for cases (a) and (b) the errors rise to approximately $15\%$ and $9\%$, respectively.


Finally, we examine reconstruction with partial data for cases (b) and (c), which have time-dependent PDE parameters. The setup mirrors \Cref{sec:pde1}: the regularization parameters $(\lambda_{~R},\lambda_{~K})$ are varied over the same set as before, with $\Delta t = 0.015625$ in this experiment. The relative errors in $~\Phi^\text{pred}$ for each pair $(\lambda_{~R},\lambda_{~K})$ are reported in \Cref{tab:eq2_reg}. In contrast to the reaction-diffusion-advection case, the errors do not blow up in the absence of regularization, although it remains considerably larger compared to $\tau_\text{sol}$: $\sim 10^{-2}$ in case (b) and $\sim 10^{-5}$ in (c). This happens due to strongly decaying memory kernels for the two cases. Setting one of the regularization parameters to zero and varying the other yields errors comparable to $\tau_\text{sol}$, whereas increasing both parameters simultaneously leads to substantially larger errors, exceeding 100\% in many instances.

\begin{table}
  \caption{Relative errors in the predicted trajectories $~\Phi^\text{pred}$ and operators $~R$ and $~K$ for various $\Delta t$ across all damped wave equation test cases. Both $~K$ and $~R$ exhibit second-order convergence and $~\Phi^\text{pred} \sim \tau_{sol}$. A zero error for $~K$ indicates it was used in \eqref{eq:errK} because $~K^\text{true}$ was unavailable. Operators were reconstructed with full data.}\label{tab:eq_2_conv}
  \centering
  \begin{tabular}{
    c
    S[table-format=1.4e-2, table-sign-exponent]
    S[table-format=1.2e-03, table-sign-exponent]
    S[table-format=1.2e-03, table-sign-exponent]
    S[table-format=1.2]
    S[table-format=1.2e-03, table-sign-exponent]
    S[table-format=1.2]
  }
    \toprule
    Case & {$\Delta t$} & {Error in $~\Phi^\text{pred}$} & {Error in $~K$} & {Rate} & {Error in $~R$} & {Rate}\\
    \midrule
    (a) & 6.2500e-02  & 2.30e-11 & 3.94e-01 & {---}  & 6.43e-01 & {---} \\
        & 3.1250e-02  & 2.29e-11 & 7.47e-02 & 2.40  & 1.05e-01 & 2.61 \\
        & 1.5625e-02  & 2.28e-11 & 1.79e-02 & 2.06  & 2.44e-02 & 2.11 \\
    \bottomrule
    (b) & 6.2500e-02  & 2.04e-12 & 4.04e-01 & {---}  & 1.12e+00 & {---} \\
        & 3.1250e-02  & 2.04e-12 & 7.92e-02 & 2.35   & 1.78e-01 & 2.66 \\
        & 1.5625e-02  & 2.04e-12 & 1.93e-02 & 2.04   & 4.09e-02 & 2.12 \\
        & 7.8125e-03  & 2.04e-12 & 0.00     & {---}  & 1.02e-02 & 2.00 \\
    \bottomrule
    (c) & 6.2500e-02  & 1.59e-12  & 3.87e-01  & {---}  & 9.07e-01 & {---} \\
        & 3.1250e-02   & 1.59e-12  & 7.55e-02  & 2.35   & 1.42e-01 & 2.68 \\
        & 1.5625e-02   & 1.60e-11  & 1.45e-02  & 2.38   & 3.24e-02 & 2.13 \\
        & 7.8125e-03  & 1.60e-11  & 0.00      & {---}  & 8.03e-03 & 2.01
  \end{tabular}
\end{table}

\begin{table}
\caption{Relative errors in the predicted trajectories $~\Phi^{\text{pred}}$ for damped wave equation test cases under a finite-memory approximation of the kernel $~K$, with $~K_n = 0$ for all $n > m$ ($m = N_T$ corresponds to the full-memory case).}\label{tab:eq2_PH}
\centering
\begin{tabular}{c| c c c c c}
\toprule
 Case & $m=0.1N_T$ & $m=0.2N_T$ & $m=0.3N_T$ & $m=0.5N_T$ & $m=N_T$ \\
\midrule
(a) & $1.49\times10^{-1}$ & $9.32\times10^{-2}$ & $6.07\times10^{-2}$ & $3.25\times10^{-2}$ & $4.67\times10^{-3}$ \\
(b) & $9.06\times10^{-2}$ & $4.57\times10^{-2}$ & $2.20\times10^{-2}$ & $7.55\times10^{-3}$ & $6.17\times10^{-3}$ \\
(c) & $4.59\times10^{-2}$ & $3.29\times10^{-2}$ & $2.70\times10^{-2}$ & $1.60\times10^{-2}$ & $2.58\times10^{-3}$ \\
\end{tabular}
\end{table}

\begin{table}
\caption{Relative errors in the predicted trajectories $~\Phi^{\text{pred}}$ for the two damped wave equation test cases with time-varying $~A(t)$. Operators $~R$, $~K$, $~B$ were reconstructed with partial data by solving the regularized problem \eqref{eq:localpartialreg} for different values of the regularization parameters $\lambda_{~R}$ and $\lambda_{~K}$.}\label{tab:eq2_reg}
\centering
\begingroup
\begin{tabular}{cc*{5}{c}}
\toprule
Case  & $\lambda_{~R}\backslash\lambda_{~K}$
   & {$0$} & {$10^{-8}$} & {$10^{-4}$} & {$10^{-2}$} & {$10^{0}$} \\
\midrule
(b) & 0        & $2.39\times10^{-2}$ & $1.34\times10^{-11}$ & $2.15\times10^{-11}$ & $1.32\times10^{-11}$ & $9.91\times10^{-12}$ \\
    & {$10^{-8}$} & $2.04\times10^{-12}$ & $ \ge 1$ & $ \ge 1$ & $ \ge 1$ & $ \ge 1$ \\
    & {$10^{-4}$} & $2.04\times10^{-12}$ & $2.72\times10^{-12}$ & $ \ge 1$ & $ \ge 1$ & $ \ge 1$ \\
    & {$10^{-2}$} & $2.04\times10^{-12}$ & $2.72\times10^{-12}$ & $1.09\times10^{-4}$ & $ \ge 1$ & $ \ge 1$ \\
    & {$10^{0}$}  & $2.04\times10^{-12}$ & $2.72\times10^{-12}$ & $1.22\times10^{-4}$ & $6.76\times10^{-3}$ & $ \ge 1$ \\
\midrule
(c) & 0        & $1.26\times10^{-5}$ & $7.92\times10^{-10}$ & $4.08\times10^{-10}$ & $2.60\times10^{-10}$ & $9.04\times10^{-10}$ \\
    & {$10^{-8}$} & $1.60\times10^{-12}$ & $ \ge 1$ & $ \ge 1$ & $ \ge 1$ & $ \ge 1$ \\
    & {$10^{-4}$} & $1.59\times10^{-12}$ & $1.80\times10^{-12}$ & $ \ge 1$ & $ \ge 1$ & $ \ge 1$ \\
    & {$10^{-2}$} & $1.59\times10^{-12}$ & $1.79\times10^{-12}$ & $5.26\times10^{-5}$ & $ \ge 1$ & $ \ge 1$ \\
    & {$10^{0}$}  & $1.59\times10^{-12}$ & $1.79\times10^{-12}$ & $5.83\times10^{-5}$ & $1.99\times10^{-3}$ & $ \ge 1$ \\
\end{tabular}
\endgroup
\end{table}

\end{appendices}

\pagebreak

\bibliographystyle{siamplain} \bibliography{references}

\end{document}